\documentclass[11pt]{amsart}
\usepackage{amsfonts, amsmath, amssymb, amscd, amsthm, graphicx, setspace, enumitem, color}
\usepackage{hyperref}
\usepackage{pinlabel}
\usepackage{mathtools}
\usepackage{graphicx}
\usepackage{stackengine}
\usepackage{extarrows}
\usepackage{bm}
\usepackage{faktor}
\usepackage{tikz-cd}
\usepackage{datetime}

\usepackage[capitalise]{cleveref} 

\newtheorem{theorem}{Theorem} [section]
\newtheorem{lemma}[theorem]{Lemma}
\newtheorem{proposition}[theorem]{Proposition}
\newtheorem{corollary}[theorem]{Corollary}

\newtheorem{example}[theorem]{Example}

\newtheorem{remark}[theorem]{Remark}

\newcommand{\mc}[0]{\mathcal}
\newcommand{\Length}{\textsc{L}}
\newcommand{\ab}{\textsf{ab}}

\theoremstyle{definition}
\newtheorem{definition}[theorem]{Definition}

\newcommand{\w}{\omega}

\newcommand\R{{\mathbb R}}
\newcommand\N{{\mathbb N}}
\newcommand\Z{{\mathbb Z}}

\DeclareMathOperator{\Aut}{Aut}

\DeclareMathOperator{\im}{im}

\newcommand{\vv}{\mathbf{v}}
\newcommand{\ww}{\mathbf{w}}
\renewcommand{\tt}{\mathbf{t}}
\newcommand{\uu}{\mathbf{u}}
\renewcommand{\aa}{\mathbf{a}}
\newcommand{\bb}{\mathbf{b}}
\newcommand{\cc}{\mathbf{c}}
\newcommand{\dd}{\mathbf{d}}

\date{\today}
\title{Effective equation solving, constraints and growth in virtually abelian groups}
\author{Laura Ciobanu, Alex Evetts, Alex Levine}
\address{Maxwell Institute of Mathematical Sciences, and Department of Mathematics, Heriot-Watt University,
Edinburgh EH14 4AS, UK}
\email{l.ciobanu@hw.ac.uk}
\address{Department of Mathematics, University of Manchester M13 9PL, UK and the Heilbronn Institute for Mathematical Research, Bristol, UK}
\email{alex.evetts@manchester.ac.uk}
\address{Department of Mathematics, University of Manchester M13 9PL, UK and the Heilbronn Institute for Mathematical Research, Bristol, UK}
\email{alex.levine@manchester.ac.uk}

\begin{document}

\begin{abstract}
In this paper we study the satisfiability and solutions of group equations when combinatorial, algebraic and language-theoretic constraints are imposed on the solutions. We show that the solutions to equations with length, lexicographic order, abelianisation or context-free constraints added, can be effectively produced in finitely generated virtually abelian groups. Crucially, we translate each of the constraints above into a rational set in an effective way, and so reduce each problem to solving equations with rational constraints, which is decidable and well understood in virtually abelian groups.

 A byproduct of our results is that the growth series of a virtually abelian group, with respect to any generating set and any weight, is effectively computable. This series is known to be rational by a result of Benson, but his proof is non-constructive.
\end{abstract}

\maketitle

\section{Introduction}

One of the most famous and longstanding open questions in theoretical computer science is whether one can solve equations with length constraints in free monoids.  The elucidation of this problem has deep implications: if undecidable, it would offer a new solution to Hilbert's 10th problem about the satisfiability of polynomial equations with integer coefficients. The resolution of this problem also has important applications in the context of string solvers (SMT) for software engineering and security analysis \cite{Amadini}. Solving equations with length and other related constraints in any structures `close' to free monoids, in the realm of monoids or groups, will shed light on this outstanding open problem. In this paper we bring the first positive results in the area, the decidability of solving equations with length constraints in virtually abelian groups, as well as further types of constraints. 

Our paper combines group theory, theoretical computer science and combinatorics to study group equations, where constraints are imposed on the solutions (see Sections \ref{sec:DP}, \ref{DP:constraints}). There is a successful line of research that considers rational constraints (\cite{Dahmani, dahmani_guirardel, eqns_hyp_grps, Levine}) when solving equations in groups, that is, when the solutions are required to belong to specified rational sets (Definition \ref{def:rat_reg_con}) but little is known for non-rational constraints. We consider this latter direction in the paper in the context of virtually abelian groups. We study constraints that in general groups are non-rational (such as context-free, length or abelianisation constraints, see Definitions \ref{def:ratCF}, \ref{def:nonrat}), that turn out, somewhat surprisingly, to be rational in virtually abelian groups. Since solving equations with rational constraints in virtually abelian groups is decidable (\cite{Levine, Dahmani}), all types of equation solving problems in this paper are decidable. However, this alone is not enough to give decidability; the explicit construction of rational sets starting from a priori non-rational specifications is also necessary. We provide such algorithms that produce rational sets from the given constraints. We do not, however, claim any complexity results or practical efficiency of our algorithms.

The paper is inspired by work in theoretical computer science, where equations with (non-rational) language theoretic, length and abelianisation constraints in free monoids have been studied intensively for decades (\cite{Abdulla, RichardBuchi1988, DayManeaWE, GarretaGray}). As already mentioned, deciding algorithmically whether a free monoid equation has solutions satisfying linear length constraints is a major open question, with both theoretical implications (if undecidable, it would offer a new solution to Hilbert's 10th problem about the satisfiability of polynomial equations with integer coefficients) and practical ones, in the context of string solvers for security analysis \cite{Amadini}. 
 For much of the last century, SAT solvers were the canonical tools to deal with the golden standard of NP-complete problems. In an SAT solver the computer program looks for a solution (i.e. values of $0$ and $1$ to substitute in for the variables) to a Boolean logic formula. Moving a step higher, both in complexity and applicability, are the SMT solvers, which have boomed over the last 10 years. There, one is given a first-order formula in a wider context, such as real numbers, integers, data structures, and most relevant for our paper, strings over finite alphabets. The question is then whether there is a solution for a first-order formula that is not just in the realm of Boolean formulas, but significantly more complex: the name `satisfiability modulo theories' (SMT) represents the fact that the solvers work `modulo' a certain theory (in a universe more complex than the Boolean one) in first-order logic with equality. In algebraic terms, this is essentially equivalent to solving equations in free monoids and related structures (as many of the constraints lead naturally to commutativity or other relations between the alphabet letters), since every string can be seen as an element over the generators of a finitely generated monoid or group.

SMT solvers have a dramatic impact on software engineering and security of Web applications: the widespread interest in cybersecurity has given strong impulse to string constraint solvers because bad string manipulations can have negative effects for Web applications developed in languages like PHP or JavaScript. The performance improvements that constraint solvers have achieved over the last 5 years have been going hand in hand with progress made by computer science theorists on tackling decades-old open questions relating to complexity and decidability of word equations in free monoids. However, many of the theoretical questions, which are still open and highly relevant to understanding the boundaries of what can be done in the solvers, need an influx of ideas from algebra and mathematics more generally: solving equations in groups while keeping the constraints that have been at the core of applications can inform the monoid equations results and feed into the circle of theory and applications complementing each other. 

In the context of non-abelian free groups and free monoids (as well as hyperbolic, right-angled Artin groups etc.), length, abelianisation and context-free constraints are not rational, and the results tend to be negative: our results for virtually abelian groups contrast those of the first author and Garreta. They showed in \cite{CiobanuGarreta} that the Diophantine Problem, that is, the question of satisfiability of equations in a group (see Section \ref{sec:DP}), with abelianisation constraints (called `abelian predicates' in that paper) for non-abelian right-angled Artin groups, or for hyperbolic groups with abelianisation rank $\geq 2$, is undecidable.

It is fairly immediate to see that in the case of abelian groups, imposing abelian or length constraints is equivalent to adding additional equations to the original system, so this is simply an instance of the Diophantine Problem in abelian groups. Imposing lexicographic order constraints on solutions to abelian group equations amounts to checking for membership in easily computable sets. However, moving from abelian to virtually abelian, to establish the rationality of the language-theoretic or algebraic constraints, and especially the explicit description of those constraints as rational sets, becomes a lot more involved. Our main result is as follows.

\begin{theorem}\label{thm:main}
	In any finitely generated virtually abelian group, it is effectively decidable whether a finite system of equations with the following kinds of constraints has solutions:
	\begin{itemize}
	\item[(i)] linear length constraints (with respect to any weighted word metric),
	\item[(ii)] abelianisation constraints,
	\item[(iii)] context-free constraints,
	\item[(iv)] lexicographic order constraints.
\end{itemize}
	\end{theorem}

	Theorem \ref{thm:main} is a consequence of two facts: (1) each of the above constraints is a rational set (Theorem \ref{prop:main}), and (2) the problem of satisfiability of finite systems of equations with rational constraints is decidable and the solutions can be effectively produced \cite{EvettsLevine, Levine, Dahmani}.

	\begin{theorem}\label{prop:main}
	In any finitely generated virtually abelian group, there is an effective way to construct a rational set from a
	 \begin{itemize}
	\item[(i)] linear length constraint (with respect to any weighted word metric) (Theorem \ref{thm:lengthrational}),
	\item[(ii)] abelianisation constraint (Theorem \ref{thm:abelianisation_cons}),
	\item[(iii)] context-free constraint (Theorem  \ref{thm:CF}),
	\item[(iv)] lexicographic order constraint (Theorem \ref{thm:lexicographic}).
\end{itemize}
	\end{theorem}

Theorem B of \cite{Levine} states that the solution sets to finite systems of equations with rational constraints can be represented by EDT0L languages (with respect to a natural normal form), and that these EDT0L languages can be explicitly constructed. We refer the reader to the survey \cite{CiobanuLevine} for background on EDT0L languages and equations in groups, and do not define or discuss EDT0L languages here, since these are not needed in the rest of the paper. Corollary 6.18 of \cite{CiobanuEvetts} implies that the set of solutions to such a system can also be expressed as an EDT0L language with respect to a geodesic normal form. Thus we obtain the following.

\begin{corollary}
The set of solutions to a finite system of equations with the kinds of constraints listed in Theorem \ref{thm:main} may be expressed as an EDT0L language in (at least) two different ways: with respect to a natural normal form, and with respect to a geodesic normal form.
\end{corollary}

Finally, a consequence of our work is Proposition \ref{prop:growth}.

\begin{proposition}\label{prop:growth}(see Corollary \ref{cor:growth})
	The weighted growth series (with respect to any finite generating set) of a finitely generated virtually abelian group \(G\) can be computed explicitly.
\end{proposition}
The fact that this growth series is a rational function was proved by Benson \cite{Benson}, but his proof does not include an algorithm to calculate the series.


\section{Preliminaries}

\subsection{Describing virtually abelian groups}

It is a standard fact that we may assume that any finitely generated virtually abelian group is a finite extension of a finitely generated free abelian group. Throughout the section we let $G$ be a finitely generated virtually abelian group with free abelian normal subgroup $A$ of finite index, and use the short exact sequence
\begin{equation}\label{ses}
	1\to A\to G\to \Delta\to 1
\end{equation}
for some finite group $\Delta$.

In any decision problem that can be asked about a group, one must first decide
what information about the group is `known'. This is frequently a presentation,
but can take other forms. In our case, we describe $G$
using a free abelian basis $B$ for $A$ and
a (finite, right) transversal, which together define a finite generating set. In
addition to this, we need to know how to multiply generators.

Before we can formally define this information, we must first define a normal
form for a virtually abelian group. We start by recalling the definition of
a normal form.

\begin{definition}
  Let \(G\) be a group and \(\Sigma\) be a generating set. A \textit{normal
  form} \(\mu\) for \(G\) is a function \(\mu \colon G \to \Sigma^\ast\) that maps
  every element \(g\) to a word \(\mu(g)\) that represents \(g\).
\end{definition}
Note that our definition implies that each element has a unique representative. We
will frequently abuse this definition, and refer to a normal form \(\mu\) to mean
the image \(\mu(G)\).

We now define the standard normal form we will be using for virtually abelian
groups.
\begin{definition}\label{def:VAtransv}
  Let \(G\) be a finitely generated virtually abelian group. Fix a finite-index free abelian
  normal subgroup \(A\), a free abelian basis \(B\) for \(A\), and a total
  order \(\leq\) on \(B\); that is, write \(B = \{b_1, b_2, \ldots, b_n\}\). Let
  \(T\) be a finite transversal of \(A\) (we choose a right transversal below, but it is not necessary to adhere to right or left since the subgroup is normal). Note that \(B \cup T\)
  generates \(G\), and every element of \(G\) can be represented uniquely in
  the form
  \[
    (b_1^\ast \cup (b_1^{-1})^\ast) \cdots (b_n^\ast \cup (b_n^{-1})^\ast) T.
  \]
  The \textit{subgroup-tranversal normal form} for \(G\), with respect to \((B,
  \leq, T)\), denoted \(\eta_{B, \leq, T}\) is the function that maps an
  element \(G\) to its unique representative word in the above set. If \(B\),
  \(\leq\) and \(T\) are implicit, we will frequently denote this using
  \(\eta\).
\end{definition}

We are now in a position where we can formally define the information we will
be using to construct these algorithms.

\begin{definition}\label{def:VA}
  Let \(G\) be a finitely generated virtually abelian group. A \textit{virtually
  abelian group description} for \(G\) comprises the following:
  \begin{enumerate}
    \item A (finite) free abelian basis \(B\) for a finite-index free abelian
      normal subgroup \(A\) of \(G\), with a total order associated with it;
    \item A finite transversal \(T\) for \(A\) in \(G\);
    \item A presentation of the finite quotient $\Delta=G/A$;
    \item The function \(f \colon (B \cup T)^{\pm} \times (B \cup T)^{\pm} \to (B^\pm)^*T
      \) that maps a pair \((a_1, a_2)\) to \(\eta(a_1 a_2)\).
      That is, \(f\) `collects' the multiplication of the generators.
  \end{enumerate}
\end{definition}

\subsection{Rational and context-free sets in groups}\label{sec:FL}

Let $\Sigma$ be a finite alphabet and let $S=\Sigma^{\pm 1}$. Suppose $G$ is a group generated by $S$, and let $\pi: S^* \rightarrow G$ be the natural projection from the free monoid $S^*$ generated by $S$ to $G$, taking a word over the generators to the element it represents in the group.

A \emph{language} is any subset of $S^*$ and is called \emph{regular} if it is recognised by a finite state automaton, as is standard (\cite{groups_langs_aut}). A \emph{context-free} language is recognised by a pushdown automaton or context-free grammar. The convention we employ here is that all context-free languages are given by context-free grammars in Chomsky normal forms  (see \cite[Section 2.6.13]{groups_langs_aut}).

 We next define sets of elements in a group that are images or preimages of regular or context-free languages over the generators of the group.

\begin{definition}\label{def:rat_reg_con}
\ \begin{enumerate}
\item[(1)]
A subset $L$ of $G$ is {\em recognisable} if the full preimage
$\pi^{-1}(L)$ is a regular subset of $S^*$.
\item[(2)] A subset $L$ of $G$ is {\em rational} if $L$ is the
image $\pi(R)$  of a regular subset $R$ of $S^*$.
\item[(3)] A subset $L$ of $G$ is {\em recognisably context-free} if
$\pi^{-1}(L)$ is a context-free subset of $S^*$.
\item[(4)] A subset $L$ of $G$ is {\em context-free} if $L$ is the
image $\pi(C)$  of a context-free subset $C \subset S^*$.

\end{enumerate}
\end{definition}
It follows immediately from the definition that recognisable subsets of $G$ are
rational. Similarly, if a set is of type (3) then it is of type (4). The type (4) sets above are sometimes called \emph{algebraic} in the literature (\cite{Berstel, Carvahlo}), but we avoid this terminology because it can be confused with `algebraic sets' in the sense of `solutions sets to equations'.

\begin{definition}\label{def:rateffcon}
	A rational set \(L \subseteq G\) is said to be \emph{effectively constructible} from some input $I$ if there exists an algorithm that, on input of $I$, produces a regular subset \(R \subseteq S^*\) such that \(\pi(R)=L\).
\end{definition}

Note that if we are given rational sets \(L\) and \(L'\) in some group, and
corresponding regular languages \(R\) and \(R'\) over the generators of \(G\),
the concatenation \(LL'\) and the union \(L \cup L'\) are effectively
constructible rational sets. If, moreover, we are given an automorphism
\(f\colon G\to G\), defined by its action on the generators, the image \(f(L)\)
is an effectively constructible rational set \cite{groups_langs_aut}.

\subsection{Semilinear sets}

\begin{definition}
	Let \(G\) be a group or monoid. A subset of \(G\) is \emph{linear} if it has the form \(aB^*=\{ab : b\in B^*\}\) for some \(a\in G\) and finite \(B\subset G\). Any finite union of linear sets is called \emph{semilinear}.
\end{definition}

Semilinear sets are most frequently studied in commutative monoids \(\N^k\) and \(\Z^k\) and there we can write the monoid operation additively, so that a linear set has the form \(a+B^*=\{a+b : b\in B^*\}\). We record here a strengthened  version of a classical result of Parikh that we will need in Section \ref{sec:CF}. Let $\Sigma=\{a_1, \dots a_k\}$ and let $\ab: \Sigma^* \mapsto \mathbb{N}^{|\Sigma|}$ be the \emph{abelianisation} or \emph{Parikh map} for free monoids:
\[\ab: \Sigma^* \mapsto \N^{|\Sigma|}\] that records the occurrences of each letter $a_i$ in a word. That is, $\ab(w)=(|w|_{a_1}, \dots, |w|_{a_k})$, where $|w|_{a_i}$ denotes the number of occurrences of letter $a_i$ in $w$.

\begin{theorem}[\cite{ParikhConstructive}, Theorem 1.1]\label{Parikh}
Let $L$ be a context-free language over $\Sigma$. Then $\ab(L)$ is a semilinear subset of $\N^{|\Sigma|}  $, and there exists a regular language $R$ over $\Sigma$ such that $\ab(L)=\ab(R)$.

Moreover, given a context-free grammar $G$ with $n$ variables, one can effectively construct a finite state automaton $M$ with $O(4^n)$ states such that the languages $L(G)$ and $L(M)$ have the same Parikh image, that is, $\ab(L(G))=\ab(L(M))$.
\end{theorem}

To prove the decidability results of Section \ref{sec:length}, we are interested in showing that certain semilinear sets can be algorithmically calculated. To make this precise, we make the following definition. In our case, the specified input will always be a virtually abelian group description along
with some form of constraint.
\begin{definition}\label{def:semieffcon}
	Let \(X\subset\N^k\) (respectively \(X\subset\Z^k\)), be a subset which we know to be semilinear. We say that \(X\) is \emph{effectively constructible} if, given a specified input, we can find elements \(a_1,\ldots,a_d\) and finite subsets \(B_1,\ldots,B_d\) of \(\N^k\) (respectively \(\Z^k\)) so that \(X=\bigcup_{i=1}^d \left(a_i+B_i^*\right)\). 
\end{definition}

\subsection{The Diophantine problem ($\mc{DP}$) in groups}\label{sec:DP}

Let $\mathbf{x}=\{X_1, \dots, X_m\}$ be a set of variables, where $m\geq 1$.
For a group $G$, a \textit{finite system of equations} in $G$ over the variables $\mathbf{x}$ is a
		finite subset $\mathcal{E}$ of the free product $G \ast F(\mathcal{\mathbf{x}})$, where
		\(F(\mathbf{x})\) is the free group on \(\mathbf{x}\). If
		\(\mathcal{E} = \{w_1, \ \ldots, \ w_n\}\), then a \textit{solution} to the
		system \(w_1 = \cdots = w_n = 1\) is a homomorphism \(\phi \colon G \ast F(\mathbf{x}) \to
		G\), such that \(\phi(w_1) = \cdots =\phi (w_n) = 1_G\) and
		 \(\phi(g) = g\) for all \(g \in G\). If $\mathcal{E}$ has a solution, then it is \emph{satisfiable}.

	\begin{example} Consider the system $\mathcal{E}=\{w_1, w_2\}\subset F(a,b) \ast F(X_1, X_2)$ over the free group $F(a,b)$, where $w_1=X_1^2(abab)^{-1}$, $w_2=X_2X_1X_2^{-1}X_1^{-1}$; we set $w_1=w_2=1$, which can be written as $X_1^2=abab, X_2X_1=X_2X_1$. The solutions are $\phi(X_1)=ab, \phi(X_2)=(ab)^k$, $k \in \Z$.
	\end{example}

For a group $G$, we say that systems of equations over $G$
are \emph{decidable} over $G$ if there is an algorithm to determine whether any
given system is satisfiable.  The question of decidability of (systems of) equations is called the
{\em Diophantine Problem} for $G$, and denoted $\mc{DP}(G)$.

\subsection{The Diophantine Problem with constraints}\label{DP:constraints}

If we ask not only whether a system of equations has solutions, but whether it has solutions that belong to certain specified sets, then we call those sets \emph{constraints} and consider the Diophantine Problem with various kinds of constraints. We start first with language-theoretic constraints:

\begin{definition}\label{def:ratCF}
Let $\mathcal E$ be a system of equations on variables $\mathbf{x}=\{X_1, \dots, X_k\}$ in a group $G$.
The \emph{Diophantine Problem with rational or recognisable or context-free constraints} (of type (3) or (4)) asks about the existence of solutions to $\mathcal E$, with some of the variables restricted to taking values in specified rational, recognisable or context-free sets, respectively. 
\end{definition}

We next attach three types of constraints that have an algebraic or combinatorial nature, rather than a language theoretic one, to the Diophantine Problem. These constraints are typically not rational in arbitrary groups, although will turn out to be rational in virtually abelian groups.

Recall that $G$ is finitely generated by $S$ and so every element $g \in G$ has a length $|g|_S$, which is the length of a shortest word $w$ representing $g$ in $G$.
The length of a solution to $\mathcal E$ with respect to $S$ refers to the length(s) of the group element(s) corresponding to the solution.

 For any group $G$, let $\ab:G\to G^{\ab}$ be the natural abelianisation map to $G^{\ab}= G/G'$, that is, the quotient of $G$ by its commutator subgroup.

 Finally, we may fix an ordering on the generating set of $G$ and consider the lexicographic or shortlex order for the group based on this. Or more specifically, we may fix an ordering on the free abelian basis of \(A\), and fix an ordering on
   \(T\) to obtain a lexicographic ordering \(\leq_\text{lex}\) on the normal
   form words. By ordering the normal form words using shortlex instead,
   we obtain another ordering \(\leq_\text{shortlex}\) (see Definition \ref{def:orders}).

\begin{definition}\label{def:nonrat}
\ \begin{enumerate}
\item 
We let $\mc{DP}(G, \Length$) denote the $\mc{DP}$ with \emph{linear length constraints}. A set of linear length constraints is a system $\Theta$ of linear integer equations and inequalities where the unknowns correspond to the lengths of solutions (with respect to $|.|_S$) to each variable $X_i \in \mathbf{x}$ (see also Definition \ref{def:length_constraint}). Then $\mc{DP}(G, \Length)$ asks whether solutions to $\mathcal E$ exist for which the lengths satisfy the system $\Theta$.

%
%

\item We write $\mc{DP}(G, \ab)$ for the $\mc{DP}$ where an abelian predicate is added, or equivalently, \emph{abelianisation constraints} are imposed. A set of abelianisation constraints is a system $\Theta$ of equations in the group $G^{\ab}$, and $\mc{DP}(G, \ab)$ asks whether a solution to $\mathcal E$ exists such that the abelianisation of the solution satisfies the system $\Theta$ in $G^{\ab}$.

\item We write $\mc{DP}(G, <)$ for the $\mc{DP}$ where an order predicate is added. A set $\mc{O}$ of order constraints consists of several order relations imposed on the solutions, and  $\mc{DP}(G, <)$ asks whether a solution to $\mathcal E$ exists that satisfies the constraints in $\mc{O}$.


\end{enumerate}
\end{definition}

\begin{example} Consider the virtually abelian group $G=\langle a, b : bab^{-1}=a^{-1}\rangle$ (this is the Klein bottle group and is virtually $\Z^2$), which has abelianisation $\Z\times \Z/{2\Z}$ and length function $|.|=|.|_{\{a,b\}}$. We study the equation $XaY^2bY^{-1}=1$ over variables $X,Y$ in $G$.
\begin{enumerate}
\item An instance of $\mc{DP}(G, \Length)$ is: decide whether there are any solutions $(x,y)$ such that $|x|=|y|+2$; the answer is yes, since $(x,y)=(b^{-1}a^{-1}, 1)$ is a solution with $|x|=|y|+2$.
\item An instance of $\mc{DP}(G, \ab)$ is: decide whether there are any solutions $(x,y)$ such that $\ab(x)=\ab(y)$ (we use additive notation for $\Z\times (\Z/{2\Z})$); the answer is no, since $\ab(xay^2by^{-1})=\ab(x)+\ab(y)+(1,\bar{1})=(0,\bar{0})$ together with $\ab(x)=\ab(y)$ lead to $2\ab(y)=(-1,\bar{1})$, which is not possible in $\Z\times \Z/{2\Z}$.
\item Fix the order $a<a^{-1}<b<b^{-1}$ and let \(\leq_\text{lex}\) represent the induced lexicographic order on all elements of $G$, i.e. pairs of elements $(g,h) \in G \times G$ such that $g\leq_\text{lex}h$ if $g$ is lexicographically smaller than $h$.

An instance of $\mc{DP}(G, \leq_\text{lex})$ is: decide whether there are any solutions $(x,y)$ such that $x\leq_\text{lex}y$; the answer is `yes' since $x=a^{-1}, y=b^{-1}$ is a solution with $x\leq_\text{lex}y$.
\end{enumerate}
\end{example}

%

\section{Semilinear sets: effective construction in free abelian groups}\label{sec:semilinear}

In this section we generalise results of Eilenberg and Sch\"utzenberger \cite{RationalSets}, and Ginsburg and Spanier \cite{GinsburgSpanier}, to show that we can effectively calculate expressions for intersections and complements of semilinear subsets of free abelian groups.
In this section we deal exclusively with semilinear subsets of the commutative monoids \(\N^k\) and \(\Z^k\) and therefore write the monoid operation additively, so that a linear set has the form \(a+B^*=\{a+b : b\in B^*\}\).

We are interested in effectively constructing certain semilinear sets (recall Definition \ref{def:semieffcon}).
\begin{remark}\label{rem:semilinearunion}
	By definition, if we are given semilinear expressions for sets \(X\) and \(Y\) in any group or monoid, then \(X\cup Y\) is a semilinear set, and we can effectively construct a semilinear expression for it: simply the union of the constituent linear sets of \(X\) and \(Y\).
\end{remark}
The analogous statements for intersection and complement require proof. In Theorem III of \cite{RationalSets}, Eilenberg and Sch\"utzenberger prove that intersections and complements of semilinear subsets of any commutative monoid are semilinear, however their proof is not constructive. On the other hand, Ginsburg and Spanier \cite{GinsburgSpanier} provide a constructive proof of the same facts for semilinear subsets of \(\N^k\). In this section we show that their results can be generalised to show that we can effectively construct expressions for intersections and complements of semilinear subsets of free abelian groups. We start with the definition of an integer affine map.
\begin{definition}
	An \emph{integer affine map} \(\phi\colon\Z^k\to\Z^l\) is a function given by an \(l\times k\) matrix \(M\) and a constant \(c\in\Z^l\), so that \(\phi(x)=Mx+c\) for all \(x\in\Z^k\). Note that this is a group homomorphism composed with a translation.
\end{definition}
\begin{lemma}\label{lem:semilinearimage}
	Let \(X\subset\Z^k\) be semilinear, and \(\phi\colon\Z^k\to\Z^l\) be an integer affine map. Then the image \(\phi(X)\) is an effectively constructible semilinear subset of \(\Z^l\).
\end{lemma}
\begin{proof}
	By Remark \ref{rem:semilinearunion}, we may assume that \(X\) is linear, say \(X=a+B^*\). We have \(\phi\colon x\mapsto Mx+c\) for some \(\l\times k\) integer matrix \(M\) and constant \(c\in\Z^l\). The image of \(X\) is then
	\begin{align*}
		\phi(X)=\{\phi(a+b) : b\in B^*\} = \phi(a)+\{Mb : b\in B^*\} = \phi(a)+\{Mb : b\in B\}^*,
	\end{align*}
	which is a linear set.
\end{proof}

We prove the following generalisation of Theorem 6.1 of \cite{GinsburgSpanier}.
\begin{theorem}\label{thm:intersectionsemi}
	Let \(X\) and \(Y\) be semilinear subsets of \(\Z^k\). Then \(X\cap Y\) is semilinear, and effectively constructible from \(X\) and \(Y\).
\end{theorem}
\begin{definition}\label{def:Dorder}
	Define a partial order on \(\N^k\) by comparing entries coordinate-wise. That is, for any pair of vectors \(u,v\in\N^k\), \(u\leq v\) if and only if \(e_i\cdot u\leq e_i\cdot v\) for each \(i\in\{1,\ldots,k\}\), where \(e_i\) is the \(i\)th standard basis vector.
\end{definition}
The next result is well-known, sometimes referred to as Dickson's Lemma.
\begin{lemma}[\cite{Dickson}, Lemma A]\label{lem:Dickson}
	For any \(U\subset\N^k\), the set of minimal elements of \(U\), with respect to the partial order in Definition \ref{def:Dorder}, is finite.
\end{lemma}
Lemma \ref{lem:Dickson} does not provide an algorithm for finding such minimal solutions in general. The following result allows us to find these solutions in a particular case. For clarity, we use bold face letters to denote vectors in \(\Z^k\).

\begin{lemma}[\cite{GinsburgSpanier}, Lemma 6.4 and Lemma 6.5]\label{lem:naturalGS}
	Fix constants \(\tt_1,\ldots,\tt_p,\uu_1,\ldots,\uu_q\in\N^k\), and \(\ww\in\Z^k\), for some \(p,q\geq0\), and consider the following identity with variables \(x_i,y_i\)
	\begin{equation*}
		\sum_{i=1}^p x_i\tt_i - \sum_{i=1}^q y_i\uu_i = \ww.
	\end{equation*}
	Then
	\begin{enumerate}
		\item it is decidable to determine whether there exists a non-zero tuple \((x_1,\ldots,x_p,y_1,\ldots,y_q)\in\N^{p+q}\) that satisfies this identity, and
		\item if the identity has any such solutions, the finite set of tuples in \(\N^{p+q}\) which are minimal (in the sense of Lemma \ref{lem:Dickson}) amongst all non-zero solutions can be effectively constructed.
	\end{enumerate}
\end{lemma}

We now prove the following adapted version of Lemma \ref{lem:naturalGS} so that we can deal with linear sets in \(\Z^k\).
\begin{lemma}\label{lem:integerGS}
	Fix constants \(\vv_1,\ldots \vv_m,\ww\in\Z^k\), for some \(m\geq0\), and consider the following identity with variables \(x_i\)
	\begin{equation*}
		\sum_{i=1}^m x_i\vv_i = \ww.
	\end{equation*}
	Then
	\begin{enumerate}
		\item it is decidable to determine whether there exists a non-zero tuple \((x_1,\ldots,x_m)\in\N^m \) that satisfies this identity, and
		\item if the identity has any non-zero solutions, the finite set of tuples which are minimal (in the sense of Lemma \ref{lem:Dickson}) amongst all non-zero solutions can be effectively constructed.
	\end{enumerate}
\end{lemma}
\begin{proof}
	For each \(i\), decompose the vector \(\vv_i=\tt_i-\uu_i\) for \(\tt_i,\uu_i\in\N^k\) and rewrite the identity as
	\begin{equation*}
		\sum_{i=1}^m x_i\tt_i - \sum_{i=1}^m x_i \uu_i = \ww.
	\end{equation*}
	A tuple \((x_1,\ldots,x_m)\) is a solution to this identity if and only if it can be extended to a solution \((x_1,\ldots,x_{2m})\) to the system of equations
	\begin{align*}
		\sum_{i=1}^m x_i\tt_i - \sum_{i=1}^m x_{m+i}\uu_i &= \ww \\
		x_i &=x_{m+i},~i\in\{1,\ldots,m\}.
	\end{align*}
	The identities between the coefficients \(x_i\) can be encoded back into a single equation as
	\begin{align*}
		\sum_{i=1}^m x_i\tilde{\tt}_i - \sum_{i=1}^m x_{m+i}\tilde{\uu}_i = \tilde{\ww},
	\end{align*}
	where \(\tilde{\tt}_i\in\N^{k+m}\), with \(\tt_i\) providing the first \(k\) coordinates, \(1\) in the \((k+i)th\) position, and zeroes elsewhere; \(\tilde{\uu}_i\) defined analogously; and \(\tilde{\ww}\) simply \(\ww\) with \(m\) extra zero entries. A tuple \((x_1,\ldots,x_m)\) solves the original identity if and only if its double \((x_1,\ldots,x_m,x_1,\ldots,x_m)\) solves this latter equation. Furthermore, such a pair of solutions are either both minimal or both non-minimal.

	Noting that each \(\tilde{\tt}_i\) and \(\tilde{\uu}_i\) have non-negative integer entries, we can now apply Lemma \ref{lem:naturalGS} to finish the proof.
\end{proof}

\begin{proof}[Proof of Theorem \ref{thm:intersectionsemi}]
We closely follow the proof of Theorem 6.1 of \cite{GinsburgSpanier}, but appealing to Lemma \ref{lem:integerGS}. Since \((X_1\cup X_2)\cap Y = (X_1\cap Y)\cup (X_2\cup Y)\), and in light of Remark \ref{rem:semilinearunion}, it suffices to prove the theorem for \(X\) and \(Y\) linear sets. So let \(X=\aa+\{\bb_1,\ldots,\bb_p\}^*\) and \(Y=\cc+\{\dd_1,\ldots,\dd_q\}^*\) for \(\aa,\bb_i,\cc,\dd_j\in\Z^k\).

Define the set of `solutions' to a linear equation that express the intersection $X \cap Y$:
\begin{align*}
	U &= \left\{(x_1,\ldots,x_{p+q})\in\N^{p+q} : \aa+\sum_{i=1}^p x_i\bb_i = \cc+\sum_{i=p+1}^{p+q} x_i\dd_i\right\}. \\
\end{align*}
Furthermore, we define a set \(V\) with the property that adding an element of \(V\) to an element of \(U\) results in another element of \(U\):
\begin{align*}
	V &= \left\{(x_1,\ldots,x_{p+q})\in\N^{p+q} : \sum_{i=1}^p x_i\bb_i = \sum_{i=p+1}^{p+q} x_i \dd_i\right\}.
\end{align*}
We claim that \(U\) is semilinear and effectively constructible.
Let \(\tau\colon\N^{p+q}\to\Z^k\) be the map given by \((x_1,\ldots,x_{p+q})\mapsto\sum_{i=1}^p x_i\bb_i\), which is (affine) linear, given by the \(k\times (p+q)\) matrix whose first \(p\) columns are the vectors \(\bb_1,\ldots,\bb_p\) and remaining \(q\) columns are filled with zeroes. We have \(X\cap Y=\{\aa+\tau(u) : u\in U\}=\aa+\tau(U)\). Then, by our claim and Lemma \ref{lem:semilinearimage}, \(X\cap Y\) is semilinear and effectively constructible, and the theorem is proved.

To prove the claim, let \(U_{\min}\) and \(V_{\min}\) be the minimal subsets of \(U\) and \(V\setminus\{\textbf{0}\}\), respectively (which are finite by Lemma \ref{lem:Dickson}). By Lemma \ref{lem:integerGS} we can compute the elements of \(U_{\min}\) and \(V_{\min}\) (using \(\aa,\bb_i,\cc,\dd_j\)). We will now show that \(U=\{u+v : u\in U_{\min},v\in V_{\min}^*\}\), proving the claim.
If \(u\in U_{\min}\) and \(v\in V_{\min}^*\) then clearly \(u+v\in U\). Conversely, suppose \(w\in U\). By definition of \(U_{\min}\), there is some \(w'\in U_{\min}\) such that \(w'\leq w\). Let \(w''=w-w'\in\N^k\). We have
\begin{align*}
	\sum_{i=1}^p w_i''\bb_i &= \sum_{i=1}^p(w_i-w_i')\bb_i = \sum_{i=1}^p w_i\bb_i-\sum_{i=1}^pw_i' \bb_i \\
	&=(\cc-\aa)+\sum_{i=p+1}^{p+q} w_i\dd_i - (\cc-\aa)-\sum_{i=p+1}^{p+q}w_i'\dd_i \\
	&=\sum_{i=p+1}^{p+q}(w_i-w_i')\dd_i = \sum_{i=p+1}^{p+q}w_i''\dd_i,
\end{align*}
and so \(w''\in V\).

Finally, we prove that \(V\subset V_{\min}^*\), which finishes the proof. We induct on the sum of entries \(\sum_{i=1}^{p+q} v_i\) of a vector \(v\in V\). First note the base case: \(V_{\min}^*\) contains the zero vector. Assume that for any \(v\in V\) with \(\sum_{i=1}^{p+q}v_i\leq k\), we have \(v\in V_{\min}^*\). Now  consider \(v\in V\) with \(\sum_{i=1}^{p+q}v_i=k+1\). There exists \(v'\in V_{\min}\) with \(v'\leq v\). As above, write \(v''=v-v'\). Clearly \(v''\in V\) too. Now \(v'\neq 0\) so \(v''\neq v\), so \(\sum_{i=1}^{p+q}v_i''<k+1\), so \(v''\in V_{\min}^*\) by the inductive hypothesis. And now \(v=v'+v''\in V_{\min}^*\).
\end{proof}

Theorem \ref{thm:intersectionsemi} has several corollaries, as follows.
\begin{corollary}\label{cor:semipreimage}
	Given a semilinear subset \(S\subseteq\Z^k\), and an integer affine transformation \(\phi \colon \Z^m\to\Z^k\), the preimage \(\phi^{-1}(S)\) is semilinear and effectively constructible.
\end{corollary}
\begin{proof}
	Let \(\theta\colon\Z^m\to\Z^{m+k}\) be the integer affine map defined by \(x\mapsto(x,\phi(x))\), whose image \(\theta(\Z^m)\) is semilinear and effectively constructible by Lemma \ref{lem:semilinearimage}. Since \(S\) is semilinear in \(\Z^k\), it is straightforward to construct \(\Z^m\times S\) as a semilinear subset of \(\Z^{m+k}\). By Theorem \ref{thm:intersectionsemi}, \(\theta(\Z^m)\cap(\Z^m\times S)\) is semilinear and effectively constructible. If we let \(p\colon\Z^{m+k}\to\Z^m\) be the integer affine map defined by projecting to the first \(m\) coordinates, we have
	\[\phi^{-1}(S) = p\left(\theta(\Z^m)\cap(\Z^m\times S)\right)\]
	and so by Lemma \ref{lem:semilinearimage} again, \(\phi^{-1}(S)\) is semilinear and effectively constructible.
\end{proof}


\begin{definition}\label{def:orthants}
	We define an \emph{orthant} in analogy with a quadrant in 2 dimensions.
	\begin{enumerate}
		\item A subset \(U\) of \(\Z^k\) will be called \emph{monotone} if, for each \(i\in\{1,\ldots,k\}\), the \(i\)th coordinates of every element of \(U\) have the same sign, or are zero.
		\item An \emph{orthant} of \(\Z^k\) is a monotone set that is maximal (by inclusion). That is, not contained in any larger monotone set. For example, if \(k=2\), an orthant is a quadrant, including the axes and the origin.
	\end{enumerate}
\end{definition}
\begin{lemma}\label{lem:Zkpartition}
	For each \(k\geq1\), \(\Z^k\) can be partitioned into \(2^k\) (effectively constructible) monotone linear subsets.
\end{lemma}
\begin{proof}
	Write \(e_1,\ldots e_k\) for the standard generators of \(\Z^k\). The proof is by induction on \(k\). For \(k=1\) we have the partition \(\Z=\{e_1\}^*\cup(e_1^{-1}+\{e_1^{-1}\}^*)\). Now assume that we have a disjoint union \(\Z^{k-1}=\bigcup_{i=1}^{2^{k-1}}(b_i+C_i^*)\) where each \(b_1+C_i^*\) is a monotone linear set. For each \(i\), define \(U^+_i=b_i+C_i^*+\{e_k\}^*=b_i+(C_i\cup\{e_k\})^*\) and \(U^-_i=b_i+C_i^*+e_k^{-1}+\{e_k^{-1}\}^*=b_i+e_k^{-1}+(C_i\cup\{e_k^{-1}\})^*\). By construction, each of these new sets is monotone and linear, and they are all disjoint. Furthermore, we have \(\Z^k=\bigcup_{i=1}^{2^{k-1}}(U_i^+\cup U_i^-)\), which is a disjoint union of \(2^{k}\) monotone linear sets.
\end{proof}
\begin{definition}\label{def:Qi}
	Let \(\bigcup_{i=1}^{2^k}Q_i\) denote any partition of \(\Z^k\) into \(2^k\) monotone linear sets. Lemma \ref{lem:Zkpartition} ensures that such a partition exists and is effectively constructible.
\end{definition}
As a corollary of Theorem \ref{thm:intersectionsemi}, a decomposition also exists for any semilinear subset of \(\Z^k\).
\begin{corollary}\label{cor:orthant}
	Let \(S\subseteq\Z^k\) be semilinear. Then we may decompose \(S\) as a finite disjoint union of (effectively constructible) monotone semilinear sets.
\end{corollary}
\begin{proof}
	Let \(S_i=S\cap Q_i\) for each \(i\). Each \(S_i\) is monotone by construction and semilinear and effectively constructible by Theorem \ref{thm:intersectionsemi}. Their union is disjoint and equal to \(S\).
\end{proof}

A further corollary of Theorem \ref{thm:complementsemi} concerns so-called \emph{polyhedral sets}, which are subsets of \(\Z^k\) defined by affine hyperplanes. They were introduced by Benson in \cite{Benson} to study the growth series of virtually abelian groups and have since proved to be an invaluable tool for other applications to growth and equations in these groups (e.g. \cite{EvettsLevine}).
\begin{definition}\label{def:PolSets}
	Let $k\in\N$, and let $\cdot$ denote the Euclidean scalar product.
	\begin{enumerate}
		\item[(i)] Any subset of $\Z^k$ of the form
		\begin{enumerate}
			\item[(1)] $\{z\in\Z^k : u\cdot z=a\}$,
			\item[(2)] $\{z\in\Z^k : u\cdot z\equiv a\mod b\}$, or
			\item[(3)] $\{z\in\Z^k : u\cdot z>a\}$
		\end{enumerate}
		for $u\in\Z^k$, $a\in\Z$, $b\in\N$, is an \emph{elementary region}, of type (1), (2), and (3) respectively;

		\item[(ii)]  any finite intersection of elementary sets will be called a \emph{basic polyhedral set};
		\item[(iii)]  any finite union of basic polyhedral sets will be called a \emph{polyhedral set}.
	\end{enumerate}
\end{definition}
\begin{proposition}\cite[Proposition 3.11]{CiobanuEvetts}\label{prop:polysemi}
	A subset of \(\Z^k\) is polyhedral if and only if it is semilinear.
\end{proposition}

 Here we note that a semilinear set is effectively constructible from its polyhedral expression.
\begin{corollary}\label{cor:polytosemi}
	Given an expression for a polyhedral set (as a finite union of finite intersections of elementary sets), we can effectively construct a semilinear expression for the same set.
\end{corollary}
\begin{proof}
	The proof of Proposition 3.11 of \cite{CiobanuEvetts} shows that any elementary set can be expressed as a semilinear set. Polyhedral sets are finite unions of finite intersections of elementary sets so the result follows from Remark \ref{rem:semilinearunion} and Theorem \ref{thm:intersectionsemi}.
\end{proof}

We need a generalisation of the following Theorem of Ginsburg and Spanier.
\begin{theorem}[\cite{GinsburgSpanier} Theorem 6.2]\label{thm:complementsemiN}
	Let \(X\) and \(Y\) be semilinear subsets of \(\N^k\). Then \(X\setminus Y\) is semilinear, and effectively constructible from \(X\) and \(Y\).
\end{theorem}
We will make use of the following elementary lemma.
\begin{lemma}\label{lem:semisubsetN}
	Suppose that \(U\) is a semilinear subset of \(\Z^k\) and \(U\subset\N^k\). Then \(U\) is a semilinear subset of \(\N^k\).
\end{lemma}
\begin{proof}
	First, suppose that \(U\) is linear. We have \(U=a+B^*\) for some \(a\in\Z^k\) and finite \(B\subset\Z^k\). If \(U\subset\N^k\) then since \(0\in B^*\) we have \(a\in\N^k\). Furthermore, if any \(b\in B\) was not in \(\N^k\), a sufficiently large multiple of \(b\) would give an element \(a+nb\in U\setminus\N^k\), and thus \(B\subset\N^k\), making \(U\) a linear subset of \(N^k\). Now suppose that \(U\) is semilinear and contained in \(\N^k\). So all of its finitely many linear components also satisfy the hypothesis, and are hence linear subsets of \(\N^k\), making \(U\) a semilinear subset of \(\N^k\).
\end{proof}

\begin{theorem}\label{thm:complementsemi}
	Let \(X\) and \(Y\) be semilinear subsets of \(\Z^k\). Then \(X\setminus Y\) is semilinear, and effectively constructible from \(X\) and \(Y\).
\end{theorem}
\begin{proof}
	First, note that any automorphism of \(\Z^n\) defined by sending each basis vector to either itself or its inverse is an integer affine map and hence if \(U\subset\Z^k\) is semilinear then its image under such an automorphism is also semilinear, and effectively constructible from \(U\).

	Suppose that \(X,Y\subset\Z^k\) are semilinear. Theorem \ref{thm:intersectionsemi} implies that, for any \(Q=Q_i\) of Definition \ref{def:Qi}, \(Q\cap X\) and \(Q\cap Y\) are semilinear and effectively constructible. There is an automorphism \(\theta\in\Aut(\Z^k)\) taking \(Q\) into \(\N^k\subset\Z^k\), and by the previous paragraph \(\theta(Q\cap X)\) and \(\theta(Q\cap Y)\) are semilinear and effectively constructible. By Lemma \ref{lem:semisubsetN}, these sets are semilinear subsets of \(\N^k\) itself, and so Theorem \ref{thm:complementsemiN} implies that \(\theta(Q\cap X)\setminus\theta(Q\cap Y)\) is semilinear and effectively constructible. Applying \(\theta^{-1}\) then gives that \((Q\cap X)\setminus(Q\cap Y)=Q\cap(X\setminus Y)\) is semilinear and effectively constructible. Finally, \(X\setminus Y\) is semilinear and effectively constructible as the union of finitely many intersections \(Q\cap (X\setminus Y)\).
\end{proof}

\section{Computing geodesics in virtually abelian groups}\label{sec:geodesics}



The set of all geodesics in a virtually abelian group has been analysed from the point of view of language theory and growth by several authors \cite{NeumannShapiro95, NeumannShapiro97}, with the results often using particular generating sets. Most recently, geodesics were studied with respect to all generating sets in \cite{Bishop}. Here we consider a subset of the entire set of geodesics, that is, we describe how to compute a \emph{geodesic normal form}, that is, one geodesic word per group element, for a finitely generated virtually abelian group with respect to any choice of (weighted) generating set. We formally define the notion of weight
below.

\begin{definition}
	Let \(G\) be a group with finite generating set \(\Sigma\) which has a weight function \(\w\colon\Sigma\to\N\setminus\{0\}\). This weight function extends to elements \(g\in G\) in the natural way:
\begin{equation*}\label{weight}
	\w(g) = \min\left\{\sum_{i=1}^k \w(\sigma_i) : \sigma=\sigma_1\cdots\sigma_k=_G g, \sigma_i \in \Sigma \right\}.
\end{equation*}
A word that has minimal weight amongst all representatives of the same group element will be called \(\emph{geodesic}\).
\end{definition}

In \cite{Benson}, Benson proved that every virtually abelian group admits a geodesic normal form (with respect to any weighted generating set) that can be described using polyhedral sets, which are semilinear by Proposition \ref{prop:polysemi}. In this section, we show that the semilinear sets in question are effectively constructible. As a consequence we see, in Corollary \ref{cor:growth}, that the growth series of the group can be effectively calculated.

\begin{theorem}\label{thm:geodesicconstruction}
	Let \(G\) be a finitely generated virtually abelian group, with a description as in Definition \ref{def:VA}. Let \(\Sigma\) be any finite generating set for \(G\), with elements of \(\Sigma\) given in subgroup-transversal normal form, equipped with a weight function \(\w\).

	Then there exists a geodesic normal form \(\mathcal{U}\subset\Sigma^*\) with a finite partition into subsets \(U\), each equipped with an injective map \(\phi\colon U\to \N^m\) (for some \(m\) depending on \(U\)) such that each image \(\phi(U)\) is semilinear and effectively constructible (from the group description). Furthermore, the weights of the elements \(u\in U\) can be easily calculated from their images \(\phi(u)\). 
\end{theorem}
In Section \ref{sec:length} we will see that Theorem \ref{thm:geodesicconstruction} allows us to solve equations with length constraints. To prove Theorem \ref{thm:geodesicconstruction}, we follow the ideas of \cite{Benson} and set-up of \cite{Evetts, CiobanuEvetts} to construct the geodesics, but ensure that each step is effective.



\begin{definition}\label{def:extendedgenset}
	Let \(G\) be virtually abelian with description as in Definition \ref{def:VA}, and write \(\Z^n\) for the normal free abelian subgroup of finite index. Define the \emph{extended generating set} \[S=\{s_1s_2\cdots s_k : s_i\in\Sigma,~1\leq k\leq [G :\Z^n]\},\] where the weight of an element of \(S\) is \(\w(s)=\sum_i\w(s_i)\), and extend this to a weight on the group as in (\ref{weight}). Observe that the weight with respect to \(S\) is equal to the weight with respect to \(\Sigma\). We work with \(S\) from now on. Words in \(S^*\) that have minimal weight amongst all words that represent the same group element will be called \emph{geodesic}.
\end{definition}

\begin{definition}[Patterns and patterned words]\label{def:patterns}
	Let \(G\) be as in the statement of Theorem \ref{thm:geodesicconstruction}, with \(\Z^n\) and \(S\) as in Definition \ref{def:extendedgenset}.
	\begin{enumerate}
		\item Define \(X=S\cap\Z^n=\{x_1,\ldots,x_r\}\) and \(Y=S\setminus X=\{y_1,\ldots,y_s\}\) (membership of \(\Z^n\) can be determined by repeated applications of the function \(f\) of Definition \ref{def:VA}). Let \(P=\{\pi\in Y^* : |\pi|_S\leq [G :\Z^n]\}\) be a finite set that we call the set of \emph{patterns}.
		\item For each pattern \(\pi=\pi_1\pi_2\cdots\pi_k\in P\), define the set of \emph{patterned words}
		\[W^\pi = \left\{x_1^{w_1}\cdots x_r^{w_r}\pi_1 x_1^{w_{r+1}}\cdots x_r^{w_{2r}}\pi_2\cdots\pi_k x_1^{w_{kr+1}}\cdots x_r^{w_{kr+r}} : w_i\in\N \right\}\subset S^*.\]
		\item For each \(W^\pi\), define the bijection \(\phi_\pi\colon W^{\pi}\to\N^{kr+r}\) which records the powers of the generators from \(X\):
		\[\phi_\pi\colon x_1^{w_1}\cdots x_r^{w_r}\pi_1 x_1^{w_{r+1}}\cdots x_r^{w_{2r}}\pi_2\cdots\pi_k x_1^{w_{kr+1}}\cdots x_r^{w_{kr+r}} \mapsto \begin{pmatrix} w_1 \\ w_2 \\ \vdots \\ w_{kr+r}\end{pmatrix}\in\N^{kr+r}.\] For brevity, write \(m(\pi)=kr+r\).
	\end{enumerate}
\end{definition}
The following lemma ensures that our set of patterns $P$ is sufficient to provide geodesics for every element of $G$.
\begin{lemma}[Proposition 11.3 of \cite{Benson}]\label{lem:finiteP}
	For each element $g\in G$, there exists some $\pi\in P$ such that $W^\pi$ contains a geodesic representative for $g$.
\end{lemma}

We now construct matrices, vectors, and integer affine maps that will allow us to move between patterned words and their corresponding subgroup-transversal normal forms.
\begin{definition}[Structure constants]\label{def:structureconstants}
	Form the \(n\times r\) matrix \[Z = \begin{pmatrix} \vert & \vert & & \vert \\ x_1 & x_2 & \cdots & x_r \\ \vert & \vert & & \vert \end{pmatrix}\] whose columns are the elements of \(X\) (expressed as column vectors in \(\Z^n\)). For \(y_k\in Y\) and a standard basis vector \(e_i\in\Z^n\), normality gives \(y_ke_iy_k^{-1}\in\Z^n\). Use the function \(f\) of Definition \ref{def:VA} to calculate each conjugate in terms of standard basis vectors. For each \(y_k\in Y\), form the \(n\times n\) matrix
	\[\Gamma_k = \begin{pmatrix}  \vert & \vert & & \vert \\ y_ke_1y_k^{-1} & y_ke_2y_k^{-1} & \cdots & y_ke_ny_k^{-1} \\ \vert & \vert & & \vert \end{pmatrix}.\]

	For each \(\pi=y_{i_1}y_{i_2}\cdots y_{i_k}\in P\), construct the \(n\times m(\pi)\) matrix
	\[\left(Z\ \vert\ \Gamma_{i_1}Z\ \vert\ \cdots\ \vert\ \Gamma_{i_1}\Gamma_{i_2}\cdots\Gamma_{i_k}Z\right)\] and write \(A_i^\pi\) for its \(i\)th row, so that we define \(n\) vectors in \(\Z^{m(\pi)}\), for each pattern.

	Each element of $G$ can be written uniquely as an element of the product \(\Z^n\cdot T\), where \(T\) is the transversal coming from the description of \(G\). For a fixed pattern \(\pi\), normality of \(\Z^n\) implies that every word in \(W^\pi\) represents an element of the coset \(\Z^n\pi\). Write \(t_\pi\in T\) for the corresponding representative (which can be calculated from the description of \(G\)), and define \(B_1^\pi,\ldots,B_n^\pi\in\Z\) so that the word \(\pi\) represents the group element \((B_1^\pi,\ldots,B_n^\pi)^\intercal t_\pi\). Thus for \(w\in W^\pi\), we have
	\[w =_G \left\{\begin{pmatrix} A_1^\pi\cdot\phi_\pi(w) \\ \vdots \\ A_n^\pi\cdot\phi_\pi(w) \end{pmatrix} + \begin{pmatrix} B_1^\pi \\ \vdots \\ B_n^\pi \end{pmatrix}\right\} t_\pi,\]
	which is in subgroup transversal normal form.

	Furthermore, let \(A_{n+1}^\pi\in\N^{m(\pi)}\) record the weights of the \(x_i\)s, repeating the sequence \(k+1\) times,
	\[A_{n+1}^\pi = \left(\w(x_1),\ldots,\w(x_r),\w(x_1),\ldots,\w(x_r),\ldots,\w(x_1),\ldots,\w(x_r)\right),\]
	and write \(B_{n+1}^\pi=\w(\pi)\). So we have \(\w(w)=A_{n+1}^\pi\cdot\phi_\pi(w)+B_{n+1}^\pi\) for \(w\in W^\pi\).	Define integer affine maps \(\mathcal{E}^\pi\colon\phi_\pi(W^\pi)\to\Z^{n+1}\) via \[\mathcal{E}^\pi\colon x\mapsto\begin{pmatrix} A_1^\pi\cdot x+B_1^\pi \\ \vdots \\ A_{n+1}^\pi\cdot x + B_{n+1}^\pi\end{pmatrix}.\]
\end{definition}

\begin{definition}[Ordering]\label{def:ordering}
	We define an order on each \(W^\pi\) which is compatible with the weight, as follows. Extend the set \(\{A_1^\pi,\ldots,A_{n+1}^\pi\}\) by choosing standard basis vectors \(A_{n+2}^\pi,\ldots,A_{K}^\pi\) so that \(\{A_1^\pi,\ldots,A_K^\pi\}\) has \(\R\)-rank equal to \(m(\pi)\).

	Then define an order on \(W^\pi\) as:
	\(v\leq_\pi w\) if and only if either \(v=w\) or there exists \(k\in[1,K]\) such that \(A_i^\pi\cdot\phi_\pi(v)=A_i^\pi\cdot\phi_\pi(w)\) for \(i<k\) and \(A_k^\pi\cdot\phi_\pi(v)<A_k^\pi\cdot\phi_\pi(w)\). If restricted to any subset of \(W^\pi\) representing a single group element, this is a total order, and indeed a well-order (the proof is straightforward and can be found in \cite{Benson}).
\end{definition}
Now we reduce the set \(W^\pi\) of all \(\pi\)-patterned words to a set \(V^\pi\) of only those words which are minimal-weight element representatives (amongst \(W^\pi\)).
\begin{definition}
	Define \[V^\pi = \{v\in W^\pi : \text{if }w\in W^\pi\text{ and }v=_Gw\text{ then }v\leq_\pi w\}.\]
\end{definition}
The following lemma demonstrates that we can construct each \(V^\pi\) as a semilinear set. The proof is a modification of that of Proposition 6.2 of \cite{Benson} to make it constructive (see also Proposition 4.7 of \cite{Evetts}).
\begin{lemma}\label{lem:translations}
	Let \(V^\pi=\{v\in W^\pi : \text{if }w\in W^\pi\text{ and }v=_Gw\text{ then }v\leq_\pi w\}\) be as above and $\phi_\pi$ as in Def. \ref{def:patterns} (3). Then \(\phi_\pi(V^\pi)\subset \N^{m(\pi)}\) is an effectively constructible semilinear set.
\end{lemma}
\begin{proof}
	Let $m(\pi)$ be as in Definition \ref{def:patterns} (3), and $K$ as in Definition \ref{def:ordering}. Define
	\begin{equation*}
		\mathcal{T}^\pi=\bigcup_{i=1}^{K-n}\left(\left(\bigcap_{j=1}^{n+i-1}\{\tau\in\Z^{m(\pi)} : A_j^\pi\cdot\tau=0\}\right)\cap\{\tau\in\Z^{m(\pi)} :  A_{n+i}^\pi\cdot\tau>0\}\right).
	\end{equation*}
	It is straightforward to show that for \(v,w\in W^\pi\), \(v\) and \(w\) represent the same group element with \(w\leq_\pi v\) if and only if there exists \(\tau\in\mathcal{T}^\pi\) with \(\phi_\pi(v)=\phi_\pi(w)+\tau\) (see Section 6 of \cite{Benson}). The set \(\mathcal{T}^\pi\) is clearly polyhedral and so, by Corollary \ref{cor:polytosemi}, we can construct a semilinear expression for it, say \(\mathcal{T}^\pi=\bigcup_{i=1}^r(a_i+B_i^*)\), for some \(a_i\in\Z^{m(\pi)}\), and finite sets \(B_i\subset\Z^{m(\pi)}\). Note also that by Corollary \ref{cor:orthant} we may assume that for each \(i\), each element of \(B_i\) is in the same orthant as \(a_i\). We claim that \(\phi_\pi(V^\pi)=\N^{m(\pi)}\setminus\bigcup_{i=1}^r(a_i+\N^{m(\pi)})\). By Theorem \ref{thm:complementsemi} we can then express this complement as a semilinear set, finishing the proof.

	To see the claim, first suppose that \(v\in V^\pi\), but also \(\phi_\pi(v)\in a_i+\N^{m(\pi)}\) for some \(i\). So \(\phi_\pi(v)=a_i+\phi_\pi(w)\) for some \(w\in\N^{m(\pi)}\), and therefore \(v\) and \(w\) represent the same group element with \(w\leq_\pi v\), but this contradicts the assumption that \(v\in V^\pi\).

	Conversely, suppose we have \(v\in W^\pi\) with \(\phi_\pi(v)\in \N^{m(\pi)}\setminus\bigcup_{i=1}^r(a_i+\N^{m(\pi)})\) but \(v\notin V^\pi\). So there exists \(\tau\in\mathcal{T}^\pi\), \(v_0\in V^\pi\), with \(\phi_\pi(v)=\phi_\pi(v_0)+\tau\) (and \(v,v_0\) representing the same group element). We have \(\tau=a_i+b\) for some \(i\) and some \(b\in B_i^*\), and so \(\phi_\pi(v)=a_i+b+\phi_\pi(v_0)\). We claim that \(\phi_\pi(v)-a_i\in\N^{m(\pi)}\), so that \(\phi(v)\in a_i+\N^{m(\pi)}\), which is a contradiction. Since \(a_i\) and \(b\) are in the same orthant, we have \(e_j\cdot a_i\leq e_j\cdot(a_i+b)\) for each standard basis vector \(e_j\). In the case \(e_j\cdot a_i\geq0\) we therefore have \(e_j\cdot a_i\leq e_j\cdot(a_i+b)\), and so \(e_j\cdot(\phi_\pi(v)-a_i)\geq e_j\cdot(\phi_\pi(v)-a_i-b)=e_j\cdot\phi_\pi(v_0)\geq0\) (since \(v_0\in V^\pi\)). In the case \(e_j\cdot a_i<0\) we have \(e_j\cdot(\phi_\pi(v)-a_i)\geq-e_j\cdot a_i\geq0\). So we have \(\phi_\pi(v)-a_i\in\N^{m(\pi)}\) as claimed.
\end{proof}

We have now constructed everything we require to prove the main theorem of this section.
\begin{proof}[Proof of Theorem \ref{thm:geodesicconstruction}]
	The sets \(V^{\pi}\) have been constructed so as to provide us with a geodesic representative for every group element (Lemma \ref{lem:finiteP} ensures that our finite set of patterns is sufficient). However, there may be group elements represented by words with more than one pattern, which will be represented in more than one of the sets \(V^{\pi}\). In order to pass to a set of \emph{unique} geodesics, and hence a normal form, we need to remove such overlaps. Let \(\mathbf{1}_i\) denote the \((2n+2)\)-dimensional vector whose \(i\)th entry is \(1\) with zeroes elsewhere. Define polyhedral sets
	\begin{align*}
		\Theta=\bigcap_{i=1}^n\left\{\theta\in\Z^{2n+2} : \theta\cdot(\mathbf{1}_i-\mathbf{1}_{i+n+1})=0\right\} \cap \left\{\theta\in\Z^{2n+2} : \theta\cdot(\mathbf{1}_{n+1}-\mathbf{1}_{2n+2})>0\right\}, \\
		\Theta_*=\bigcap_{i=1}^n\left\{\theta\in\Z^{2n+2}: \theta\cdot(\mathbf{1}_i-\mathbf{1}_{i+n+1})=0\right\} \cap \left\{\theta\in\Z^{2n+2} : \theta\cdot(\mathbf{1}_{n+1}-\mathbf{1}_{2n+2})=0\right\}.
	\end{align*}
	By Corollary \ref{cor:polytosemi} we can rewrite \(\Theta\) and \(\Theta_*\) as semilinear subsets of \(\Z^{2n+2}\). Write \(p\colon\Z^{2n+2}\to\Z^{n+1}\) for projection onto the first \(n+1\) coordinates (which is an integer affine map). Now define
	\begin{align*}
		R^{\pi,\mu} = (\mathcal{E}^\pi)^{-1}\circ p\left[\left(\mathcal{E}^\pi\circ\phi_\pi(V^\pi)\times\mathcal{E}^\mu\circ\phi_\mu(V^\mu)\right)\cap\Theta\right], \\
		R_*^{\pi,\mu} = (\mathcal{E}^\pi)^{-1}\circ p\left[\left(\mathcal{E}^\pi\circ\phi_\pi(V^\pi)\times\mathcal{E}^\mu\circ\phi_\mu(V^\mu)\right)\cap\Theta_*\right]
	\end{align*}
	for each \(\pi,\mu\in P\). By Theorem \ref{thm:intersectionsemi}, Lemma \ref{lem:semilinearimage} and Corollary \ref{cor:semipreimage}, each of these can be explicitly expressed as semilinear sets. Fix an arbitrary total order on the set of patterns \(P=\{\pi_1,\pi_2,\ldots,\pi_{|P|}\}\). For each pattern, define
	\begin{equation*}
		Q^{\pi_k}=\phi_{\pi_k}(V^{\pi_k})\setminus \left(\bigcup_{i\neq k} R^{\pi_k,\pi_i}\cup \bigcup_{j<k}R_*^{\pi_k,\pi_j}\right).
	\end{equation*}
	Theorem \ref{thm:complementsemi} implies that each \(Q^{\pi_k}\) is an effectively constructible semilinear set. Finally, define the set of geodesics \(U^{\pi_k}=\phi_{\pi_k}^{-1}(Q^{\pi_k})\). By construction, the sets \(U^\pi\) provide a complete set of unique geodesics for the elements of \(G\). See also Section 12 of \cite{Benson}, or the proof of Theorem 4.2 of \cite{Evetts}. Furthermore, the weight of \(u\in U^{\pi}\) is related to that of its image by \(\w(u)=A_{n+1}^\pi\cdot\phi_\pi(u)+B_{n+1}^\pi\) (recalling Definition \ref{def:structureconstants}). This proves the theorem.
\end{proof}

\subsection{Growth Series}
This section discusses a corollary of Theorem \ref{thm:geodesicconstruction} that does not relate to solving equations. We have the following slight generalisation of the standard definition of growth.
\begin{definition}
	Let \(G\) be a group with finite weighted generating set \(S\). The (spherical) weighted growth function is given by \(\sigma_{G,S}(n)=\#\{g\in G\colon \w(g)=n\}\), where \(\w(g)\) denotes the weight of \(g\) with respect to \(S\). The weighted growth series of \(G\) is then the generating function
	\[\sum_{n=0}^\infty \sigma_{G,S}(n)z^n\in\Z[[z]].\]
\end{definition}
\begin{definition}
	The set of \emph{\(\N\)-rational functions} is the smallest set of functions \(f(z)\) containing all polynomials in \(z\) and closed under multiplication, addition, and quasi-inverse (that is, if \(f(z)\) is in the set with \(f(0)=0\), then its quasi-inverse \(\frac{1}{1-f(z)}\) is in the set).
\end{definition}
The main result of \cite{Benson} is that the weighted growth series of a virtually abelian group is an \(\N\)-rational function, but the proof is non-constructive. Explicitly constructing the semilinear sets describing the geodesics allows us to calculate the growth series. More precisely, we have the following corollary of Theorem \ref{thm:geodesicconstruction}.
\begin{corollary}\label{cor:growth}
	The (\(\N\)-rational) weighted growth series of a finitely generated virtually abelian group can be explicitly computed.
\end{corollary}
To establish Corollary \ref{cor:growth}, we need the following lemma that follows from the proof of \cite[Theorem IV]{RationalSets}.
\begin{lemma}(see \cite[Theorem IV]{RationalSets})\label{lem:linearindependence}
	If \(X\subset\N^k\) is semilinear, it can be expressed as a disjoint union of finitely many linear sets \(X=\bigcup_{i=1}^r(a_i+B_i^*) \) such that each \(B_i\) consists of linearly independent elements. Furthermore, \(a_i\) and \(B_i\) can be effectively computed from the original semilinear expression for \(X\).
\end{lemma}
%
Lemma \ref{lem:linearindependence} allows us to make the following proposition, appearing in \cite{CiobanuEvetts}, effective. See also \cite{DAIV} for more work on the growth series of semilinear sets.
\begin{proposition}\label{prop:semilineargrowth}
	If \(X\subset\N^k\) is semilinear, its growth series (with respect to a weighted \(\ell_1\)-norm) is \(\N\)-rational and can be explicitly computed.
\end{proposition}
\begin{proof}
	Lemma \ref{lem:linearindependence} implies that \(X\) can be effectively expressed as the disjoint union of finitely many linear sets of the form \(a_i+B_i^*\), where each \(B_i\) is a linearly independent set. Such a linear set has growth series \[z^{\w(a_i)}\prod_{b\in B_i}\frac{1}{1-z^{\w(b)}}.\] The growth series of \(X\) is then the sum of rational sets of the given form.
\end{proof}
Finally, we can add together the rational growth series from each disjoint piece of \(\mathcal{U}\) to obtain the growth series of the whole.
\begin{proof}[Proof of Corollary \ref{cor:growth}]
	The growth series of the group is equal to the growth series of the set of geodesics \(\mathcal{U}\). We have
	\begin{align*}
		\sum_n \sigma_{\mathcal{U}}(n)z^n = \sum_{\pi}\sum_n \sigma_{U^\pi}(n)z^n = \sum_{\pi} z^{\w(\pi)}\sum_n \sigma_{\phi_\pi(U^\pi)}(n)z^n
	\end{align*}
	so it suffices to calculate the growth series of each semilinear set \(\phi_\pi(U^{\pi})\). This can be done effectively by Proposition \ref{prop:semilineargrowth}.
\end{proof}

In future work we intend to extend the results of this section to other growth series associated to the group, as studied in \cite{Evetts} and \cite{CiobanuEvetts}.

\section{Length constraints in virtually abelian groups}\label{sec:length}

In this section we show that tuples of elements in a virtually abelian group $G$ whose word lengths satisfy systems of linear (in)equalities form rational subsets of $G$. Moreover, these rational sets can be described explicitly in terms of effectively constructible semilinear subsets of \(\Z^n\) (which is, as usual, the finite index normal free abelian subgroup of \(G\)).

\begin{definition}\label{def:length_constraint}
	Let \(G\) be a finitely generated group, with finite generating set \(S\). For any positive integer \(k\), a \emph{length constraint} on \(G^k\) will be a subset of \(G^k\) in one of the following forms
	\begin{align*}
		\left\{(g_1,\ldots,g_k)\in G^k : \sum_{i=1}^k\alpha_i|g_i|_S = \beta\right\},\\
		\left\{(g_1,\ldots,g_k)\in G^k : \sum_{i=1}^k\alpha_i|g_i|_S \leq \beta\right\},
	\end{align*}
	where \(\alpha_i,\beta\in\Z\), and \(|g|_S\) denotes word length of \(g\in G\) with respect to \(S\).
\end{definition}

The following Theorem asserts that we can explicitly compute a rational expression for any length constraint, using the sets of geodesics \(U^\pi\), constructed in Section \ref{sec:geodesics}.
\begin{theorem}\label{thm:lengthrational}
	Any length constraint (as in Definition \ref{def:length_constraint}) in a virtually abelian group can be effectively expressed as a semilinear, and hence rational, subset.
\end{theorem}
\begin{proof}
	As usual, let \(G\) be a finitely generated virtually abelian group described as in Definition \ref{def:VA}, and recall the geodesic normal form \(\mathcal{U}=\bigcup U^\pi\) constructed above. Throughout this proof we will use the notation of Section \ref{sec:geodesics}.

			Suppose we have a length constraint \(L\subset G^k\) of the first form in the definition. The same argument will work for the second form. We may obtain a complete set of unique geodesics for \(L\) by decomposing it according to which pattern the element geodesics fall into. That is, if an element $g\in G$ has its unique geodesic representative $u \in S^*$ in the set $U^{\pi_k}$ (where $1\leq k \leq |P|$ as in Definition \ref{def:patterns}), then we work with $\w(u)$ and the corresponding quantities related to the pattern $\pi_k$. Considering \(k\)-tuples of geodesics, we have
		\begin{align*}
			&\bigcup_{(\pi_1,\ldots,\pi_k)\in P^k}\left\{(u_1,\ldots,u_k) : u_i\in U^{\pi_i},~ \sum_{i=1}^k\alpha_i\w(u_i)=\beta \right\} \\
			&=\bigcup_{(\pi_1,\ldots,\pi_k)\in P^k}\left\{(u_1,\ldots,u_k) : u_i\in U^{\pi_i},~ \sum_{i=1}^k\alpha_i(A^{\pi_i}_{n+1}\cdot \phi_{\pi_i}(u_i)+B^{\pi_i}_{n+1})=\beta \right\},
		\end{align*}
		where we use the vectors \(A^\pi_{n+1}\) and integers \(B^\pi_{n+1}\) of Definition \ref{def:structureconstants} to construct the weight of the words \(u_i\).
		The summation condition can be rewritten as the identity
		\begin{equation*}
			\left(\alpha_1 A^{\pi_1}_{n+1}, \cdots \alpha_kA^{\pi_k}_{n+1}\right)\cdot\left(\phi_{\pi_1}(u_1), \cdots, \phi_{\pi_k}(u_k) \right)
			 = \beta -\sum_{i=1}^k\alpha_iB^{\pi_i}_{n+1},
		\end{equation*}
		where we have concatenated vectors on the left hand side, resulting in the scalar product of two vectors in \(\Z^{\sum_i m(\pi_i)}\).
		Thus the following expression defines the set of \(k\)-tuples of vectors corresponding to unique geodesic representatives for group elements satisfying the length constraint:
		\begin{equation*}
			M_{(\pi_1,\ldots,\pi_k)}=\left\{x\in\Z^{\sum_i m(\pi_i)} : \left(\alpha_1 A^{\pi_1}_{n+1}, \cdots \alpha_kA^{\pi_k}_{n+1}\right)\cdot x = \beta - \sum_{i=1}^k\alpha_iB_{n+1}^{\pi_i} \right\} \cap \phi_{\pi_1}(U^{\pi_1})\times\cdots\times\phi_{\pi_k}(U^{\pi_k}).
		\end{equation*}
		By Theorem \ref{thm:geodesicconstruction}, and Corollary \ref{cor:polytosemi}, we can explicitly construct this semilinar subset, for each of finitely many \(k\)-tuples in \(P^k\).

		Write \(p^{\pi_j}\colon \Z^{\sum_i m(\pi_i)}\to\Z^{m(\pi_j)}\) for the projection onto just the components corresponding to \(\pi_j\). Recall the integer affine maps \(\mathcal{E}^\pi\colon\phi_\pi(W^\pi)\to\Z^{n+1}\), and define $\tilde{\mathcal{E}}^\pi\colon\phi_\pi(W^\pi)\to\Z^n$ to be the projection onto the first $n$ coordinates, so that $\tilde{\mathcal{E}}^\pi\circ\phi_\pi(w)\cdot t_\pi =_G w$ for any $w\in W^\pi$. We then have
		\begin{align*}
			L &= \left(\tilde{\mathcal{E}}^{\pi_1}\circ\phi_{\pi_1}\circ p^{\pi_1}(M_{(\pi_1,\ldots,\pi_k)})\right)t_{\pi_1}\times\cdots\times \left(\tilde{\mathcal{E}}^{\pi_k}\circ\phi_{\pi_k}\circ p^{\pi_k}(M_{(\pi_1,\ldots,\pi_k)})\right)t_{\pi_k} \\
			&=  \left(\tilde{\mathcal{E}}^{\pi_1}\circ\phi_{\pi_1}\circ p^{\pi_1}(M_{(\pi_1,\ldots,\pi_k)})\times\cdots\times \tilde{\mathcal{E}}^{\pi_k}\circ\phi_{\pi_k}\circ p^{\pi_k}(M_{(\pi_1,\ldots,\pi_k)})\right)\cdot t_{\pi_1}\cdots t_{\pi_k}.
		\end{align*}
		Since each factor \(\tilde{\mathcal{E}}^{\pi_i}\circ\phi_{\pi_i}\circ p^{\pi_i}(M_{(\pi_1,\ldots,\pi_k)})\) is semilinear, the cartesian product is also semilinear, and hence so too is \(L\) itself. Furthermore, we have an explicit semilinear expression for \(L\), which is of course a rational expression, as required.
\end{proof}

\section{Abelianisation constraints}

 	We start with the following fact about abelianisation constraints (in any
 group) which allows us to reduce the question to checking membership in cosets
 of the commutator subgroup.

 \begin{lemma}
 	\label{lem:ab-const-form}
 	Let \(G\) be a group and \(\mathcal{E}\) be a system of equations in \(G\)
 	with abelianisation constraints. Then there exists a
 	system of equations \(\bar{\mathcal E}\) in \(G\) with the same set of
 	solutions as \(\mathcal{E}\) which has no abelianisation constraints, but
 	 instead the constraints:
 	\[
 	X_i \in [G, G]h_i,
 	\]
 	where each \(X_i\) is a variable and each \(h_i \in G\).
 \end{lemma}

 \begin{proof}
  Let \(\Theta\) denote the system of equations in \(G^\ab\) that are the
  abelianisation constraints in \(\mathcal{E}\). For each equation in
  \(\Theta\) we have two cases: the equation has one variable, occurring only
  once, or the equation has at least two variables, or more than one occurrence
  of the same variable. In the first case, the
  equation is immediately equivalent to \(X^\ab =  g^\ab\) for some \(g \in G\). This
  is equivalent to \(X \in [G, G]g\).
  Thus we can start defining \(\bar{\mathcal E}\) to have the same equations
  and abelianisation constraints as \(\mathcal E\), except with the
  constraints \(X \in [G, G]g\) replacing the single-variable, single
  occurrence abelianisation constraints.

  For the remaining abelianisation constraints that have two or more
  occurrences of variables, note that (after grouping variables) they are
  equations of the form \((X_1^\ab)^{i_1} \cdots (X_k^\ab)^{i_k} = h^\ab\)
  within \(G^\ab\), where \(X_1, \ldots, X_k\) are variables and \(h \in G\).
  We can therefore define a new variable \(Y\) in \(G\) for each such equation
  as \(Y = X_1^{i_1} \cdots X_k^{i_k}\). Then this abelianisation constraint is
  equivalent to \(Y^{\ab} = h^{\ab}\) within \(G^\ab\), which is the same as
  \(Y \in [G, G] h\). Thus we can replace each such abelianisation constraint
  in \(\mathcal E\) with a new variable, an additional equation, and a
  constraint of the form \(Y \in [G, G] h\) to define \(\bar{\mathcal E}\),
  which has the desired form.
 \end{proof}

 In light of the above lemma, to understand solutions to systems of equations
 with abelianisation constraints, we need to understand the cosets of the
 commutator subgroup. In the case of virtually abelian groups, we show that
 these are in fact rational sets, and provide an algorithm for effectively
 computing a rational expression for a coset from the virtually abelian group
 description and any coset representative.

 We start by showing that the commutator subgroup of a virtually abelian group
 is itself virtually abelian. This itself is not new; it follows from the fact
 that subgroups of virtually abelian groups are virtually abelian. We include a
 proof for completeness, and so that we can explicitly understand the index of
 abelian normal subgroups.

 \begin{lemma}
 	Let \(G\) be a finitely generated virtually abelian group, with a
 	finite-index normal free abelian subgroup of index \(k\). Then
 	\([G, \ G]\) is finitely generated virtually abelian, and admits
 	a free abelian normal subgroup of index at most \(k\).
 \end{lemma}

 \begin{proof}
 	Let \(A\) be a free abelian normal subgroup of \(G\) of index \(k\). Note
 	that \(A\) is finitely generated. As subgroups of finitely generated free
 	abelian groups are finitely generated free abelian, we have that \(A \cap
 	[G, \ G]\) is finitely generated free abelian. Moreover, using the second
 	isomorphism theorem, we have that
 	\[
 	\left|\faktor{[G, \ G]}{A \cap [G, \ G]} \right| =
 	\left| \faktor{A [G, \ G]}{A} \right| \leq \left| \faktor{G}{A} \right|.
 	\]
 	Thus the index of \(A \cap [G, \ G]\) in \([G, \ G]\) is bounded above by
 	the index $k$ of \(A\) in \(G\).
 \end{proof}

 Now we construct a diagram of maps which we will use to
 compute the rank of the commutator subgroup of \([G, G]\) from the rank of
 \(G\). We show below that the diagram exists and is exact. Note that we do
 not need the diagram to commute, so we do not attempt to show that this is the
 case.

 \begin{lemma}
 	Let \(G\) be a virtually abelian group with a
 	finite-index free abelian normal subgroup \(A\), and let \(\Delta\) be
 	the quotient of \(G\) by \(A\). Then there exist homomorphisms
 	\(\phi, \ \psi, \ \theta\) such that the diagram in Figure
 	\ref{fig:VA_commutator} is exact. The remaining homomorphisms are the
 	natural inclusions and projections.
 	\begin{figure}
 		\label{fig:VA_commutator}
 		\caption{Exact diagram from virtually abelian group}
 		\[
 		\begin{tikzcd}
 			& 1 \arrow[d] & 1 \arrow[d] & 1 \arrow[d] & \\
 			1 \arrow[r] & A \cap \left[G,  G \right] \arrow[d, hook] \arrow[r, hook]
 			& \left[G, G \right] \arrow[d, hook] \arrow[r, two heads, "\phi"]
 			& \left[\Delta,  \Delta \right] \arrow[d, hook]
 			\arrow[r] & 1 \\
 			1 \arrow[r] & A \arrow[d, two heads] \arrow[r, hook] & G \arrow[d, two heads]
 			\arrow[r, two heads, "p"]
 			& \Delta \arrow[d, two heads] \arrow[r] & 1 \\
 			1 \arrow[r] & A / (A \cap \left[G,  G\right]) \arrow[d] \arrow[r, "\psi", hook]
 			& G / \left[G, G\right] \arrow[d] \arrow[r, two heads, "\theta"]
 			& \Delta / \left[\Delta, \Delta \right] \arrow[d]
 			\arrow[r] & 1 \\
 			& 1 & 1 & 1 &
 		\end{tikzcd}
 		\]
 	\end{figure}
 \end{lemma}

 \begin{proof}
 	(\(\phi\)): Let \(p \colon G \to \Delta\) be the natural projection.  We start by
 	defining \(\phi \colon [G, G] \to [\Delta, \Delta]\) as \(\phi([g_1,
 	h_1] \cdots [g_k, h_k]) = [p(g_1), p(h_1)] \cdots [p(g_k), p(h_k)]\); that is,
  \(\phi\) is the restriction of \(p\) to \([G, G]\). As the restriction of
  an epimorphism, \(\phi\) is also an epimorphism, and since \(\ker p = A\), it follows
  that \(\ker \phi = A \cap [G, G]\).

%
 	(\(\psi\)): Let \(\psi \colon A /  (A \cap [G, G]) \to G / [G, G]\) be
 	defined by \((A \cap [G, G])g \mapsto [G, G]g\) . We first show that
 	\(\psi\) is well-defined. Let \((A \cap [G, G])g = (A \cap [G, G])h
 	\in A/ A \cap [G, G]\). Then \(gh^{-1} \in A \cap [G, G]\), and so
 	\([G, G]g = [G, G]h\), and \(\psi\) is well-defined.

 	To see that \(\psi\) is injective, let \((A \cap [G, G])g \in \ker \psi\).
 	Then \([G, G]g = [G, G]\), and so \(g \in [G, G]\). Since \(g \in A\), we
 	have that \(g \in A \cap [G, G]\), and so \((A \cap [G, G])g = A \cap [G,
 	G]\).

 	(\(\theta\)): Let \(\theta \colon G / [G, G] \to \Delta / [\Delta, \Delta]\) be defined by
 	\([G, G] g \to [\Delta, \Delta] p(g)\). To see that \(\theta\) is well-defined, note
 	that if \([G, G]g = [G, G]h \in G / [G, G]\), then \(gh^{-1} \in [G, G]\),
 	and so \(p(g)(p(h))^{-1} \in [\Delta, \Delta]\) and \([\Delta, \Delta] p(g) = [\Delta, \Delta] p(h)\).
 	Since \(p\) is surjective, we have that \(\theta\) is. As \(p\) is a
 	homomorphism, \(\theta\) is.

 	To see that \(\ker \theta = \im \psi\), first note that
 	\(\im \psi = \{[G, G] g :  g \in A\}\). It is clear that if \(g \in A\),
 	then \(g \in \ker p\), and so \([G, G] p(g) = [G, G]\). Thus
 	\(\im \psi \subseteq \ker \theta\). Now let \([G, G]g \in \ker \theta\).
 	Then \([\Delta, \Delta] p(g) = [\Delta, \Delta]\) and so \(p(g) \in [\Delta, \Delta]\). Since \(p(g) \in
 	[\Delta, \Delta] \subseteq \Delta\), \(g \in A\), and so \([G, G]g \in \im \psi\).
 \end{proof}

 Using the exact diagram, we can compute the rank of the commutator subgroup of
 a virtually abelian group, which in turn allows us to find a finite
 generating set for the commutator subgroup. An immediate corollary of this is
 that the commutator subgroup is a rational set, and so every coset of it is.

 \begin{proposition}
 	\label{prop:VA_commuator_gen_set}
 	There is an algorithm that takes as input a finitely generated virtually
 	abelian group description (as in Definition \ref{def:VA}), and outputs
 	a finite generating set for the commutator subgroup.
 \end{proposition}

 \begin{proof}
 	Let \(G\) be a finitely generated virtually abelian group, and let \(A\) be
 	a finite-index normal free abelian subgroup of \(G\), and \(T\) be a right
 	transversal for \(A\) in \(G\). Consider the diagram shown in Figure
 	\ref{fig:VA_commutator}. First note that we can construct a presentation for
 	\(G\) using the transversal and free abelian generators. Using this presentation, we
 	can compute a presentation for the abelianisation \(G / [G, G]\), by adding
 	the commutator relation for each pair of generators. Using the fact that \(G
 	/ [G, G]\) is abelian, we can standardise this presentation by grouping
 	occurences of each generator within each non-commutator relation. By
 	removing redundent relations, we can count the number of torsion-free
 	generators to compute the torsion-free rank \(m\) of \(G / [G, G]\).

 	Note that as \(\Delta / [\Delta, \ \Delta]\) is finite, the Rank-Nullity Theorem tells us
 	that the torsion-free rank of \(A / (A \cap [G, G])\) must also be \(m\).
 	Let \(n\) be the free abelian rank of \(A\). Again using the Rank-Nullity
 	Theorem, we have that the free abelian rank of \(A \cap [G, G]\) must be
 	\(n - m\).

 	We now compute a free abelian basis for \(A \cap [G, G]\) by initially
 	setting our basis set to be \(B = \emptyset\), and then iterating through
 	elements of \(A\), represented (uniquely) by shortlex-minimal geodesic
 	words in their shortlex order and adding an element to \(B\) whenever it
 	both lies in \([G, G]\) and is linearly independent from all the elements
 	of \(B\). We terminate this once we have \(n - m\) such elements. We will
 	then have a set of minimal elements of \(A\) that lie in \(A \cap [G, G]\)
 	and are all linearly independent, and thus this must be a basis for
 	\(A \cap [G, G]\). We now have a (finite) generating set \(B\) for \(A \cap
 	[G, G]\).

 	We now simply need to compute a (right) transversal for \(A \cap [G, G]\)
 	in \([G, G]\). Note that this will be finite, as \([\Delta, \Delta]\) is. Since \(\Delta\)
 	is finite, we can compute the right Cayley graph from the products of
 	elements of \(T\) (which we know). To choose a transversal for \(A \cap [G,
 	G]\), we simply need to trace one path in the Cayley graph of \([\Delta, \Delta]\)
 	from \(1_{[\Delta, \Delta]}\) to every state, and take our transversal \(U\) to be to
 	be the set of the natural lift of each word traced. We have now computed
 	\(U \cup B\), which is a finite generating set for \([G, G]\).
 \end{proof}

 Now that we know that the commutator subgroup is rational, it follows that
 all its cosets are rational, and so abelianisation constraints (when converted
 into the form of Lemma~\ref{lem:ab-const-form}) are instances of rational
 constraints.

 \begin{theorem}
 	\label{thm:abelianisation_cons}
 	If \(G\) is a finitely generated virtually abelian group with a finite generating
 	set \(\Sigma\), then every abelianisation constraint is a rational subset of
 	\(G\).

 	Moreover, a rational expression for an abelianisation constraint can be
 	computed from any coset representative \(h\) of the abelianisation constraint.
 \end{theorem}

 \begin{proof}

 	To effectively compute a rational expression for \([G, \ G]h\), we simply
 	need to compute a monoid generating set for \([G, \ G]\), which can be done
 	using the finite generating set from \cref{prop:VA_commuator_gen_set} and
 	closing it under inverses (if necessary).
 \end{proof}

\section{Context-free sets in virtually abelian groups}\label{sec:CF}

In this section we consider context-free sets (see Definition \ref{def:rat_reg_con}(3) and (4)) in virtually abelian groups. In some of the literature \cite{Herbst, Carvahlo} our `recognisably context-free sets' are called simply `context-free', while our `context-free sets' are called `algebraic'. However, we avoid the terminology `algebraic', as this can be confused with `algebraic set' in the sense of `solution set to an equation' in a group. Since, using our terminology, a recognisably context-free set is context-free, all results here will hold for recognisably context-free sets automatically because we prove them for context-free sets, and we do not refer to them additionally.

It is known that context-free sets are in fact rational in virtually abelian groups (\cite{Herbst, Carvahlo}). This can be seen in two steps. First, context-free sets in finitely generated free abelian groups coincide with rational sets via semilinear sets, by Parikh's theorem (see for example \cite{groups_langs_aut}, Theorem 2.6.23). Then every context-free set $C$ in a virtually abelian group $G$ is the finite union of context-free sets in the cosets of a free abelian finite index subgroup (see for example \cite[Prop. 3.11]{Carvahlo}); since the latter are rational by Parikh's result, the set $C$ in $G$ is rational. We make all this effective in Theorem \ref{thm:CF}.

 Since context-free sets are rational in virtually abelian groups, the Diophantine Problem with (recognisably) context-free constraints reduces (again) to a known decidable problem: the Diophantine Problem with rational constraints in virtually abelian groups \cite{Levine}. We'd like to contrast this to the Diophantine Problem with context-free constraints in other classes of groups, where imposing such constraints leads immediately to undecidability; this is because the context-free subset membership problem on its own is undecidable.
For example, the context-free set membership problem is undecidable in groups containing a free non-abelian subgroup. This follows from the result below, linking membership in context-free sets to membership in rational sets of direct products, together with the fact that membership in finitely generated, and thus rational, subgroups of $F_2 \times F_2$ is undecidable by Mihailova \cite{Mihailova}.

\begin{theorem}\cite[Corollary 6.3]{KSS}
The context-free membership problem is
decidable in a group $G$ if and only if the rational subset membership is
decidable in $G\times F_2$, where $F_2$ is the free group on two generators.
\end{theorem}


%
%


 The convention we employ here is that all context-free languages are given by context-free grammars in Chomsky normal forms (see, for example, \cite[Section 2.6.13]{groups_langs_aut}). In the proof of Theorem \ref{thm:CF} we will need to translate certain languages and preserve their complexity, and we achieve this using transducers, which we review below following \cite{KSS}.

\begin{definition}
Let $\Sigma$ and $\Omega$ be two finite alphabets. A finite automaton over $\Sigma^* \times \Omega^*$ is called a \emph{finite transducer} from $\Sigma^*$ to $\Omega^*$. A subset recognised by a finite transducer, that is, a rational set of $\Sigma^* \times \Omega^*$, is a \emph{finite transduction}. If $\rho \subseteq \Sigma^* \times \Omega^*$ is a transduction and $L \subseteq \Sigma^*$, then the \emph{image of $L$ under $\rho$} is the language $$\rho(L)=\{v \in \Omega^* : (u,v)\in \rho \textrm{\ for some\ } u\in L\} \subseteq \Omega^*.$$
\end{definition}

One can think of a transducer as a finite state automaton where the labels on the edges are pairs of words from $\Sigma^* \times \Omega^*$ instead of single letters as in a standard automaton. Then a path's label in the transducer can be seen as the graph of a function mapping words in $\Sigma^*$ to words in $\Omega^*$.

\begin{proposition}\cite[Ch.III,~Corollary 4.2]{Berstel}
Each rational transduction preserves rational and context-free languages. That is, if $L$ is a language over a finite alphabet $\Sigma$, then for each rational transduction $\rho$, $\rho(L)$ is rational if $L$ is rational, and $\rho(L)$ is context-free if $L$ is context-free.
\end{proposition}

\begin{lemma} \cite[Lemma 3.3]{KSS}\label{transducerKSS}
Let $G$ be a finitely generated group with a finite-index free normal subgroup $A$, and let $X$ and $Y$ be finite generating sets of $G$ and $A$, respectively.

\begin{itemize}
\item[1.] Then there is an explicit, computable, rational transduction $\sigma\subseteq Y^* \times X^*$ such that $W_X(h)=\sigma(W_Y(h))$, for every $h \in A$, where $W_Y(h)$ and $W_X(h)$ are the sets of all words that represent element $h$ in $A$ over $Y$ and $X$, respectively.
\item[2.] Furthermore, $\sigma^{-1}: X^* \rightarrow Y^*$ is a partial function such that, for $w \in X^*$, $\sigma^{-1}(w)$ is defined if and only if $w$ represents an element of $A$, in which case $\sigma^{-1}(w)$ is an element of $Y^*$ representing the same element of $A$ as $w$.

\end{itemize}
\end{lemma}

The transducer $\sigma$ in Lemma \ref{transducerKSS} is based on the Schreier coset graph of $G/A$. The vertices correspond to the finitely many cosets $At_i$ of $A$ in $G$ with $t_i$ in the transversal $T$, and $A$ is both initial and final state. Edges have the form $(At_i,At_ix )$, for $x \in X$, with labels $(w_{t_i,x}, x)$, where $w_{t_i,x} \in Y^*$ is a word representing the unique element $g \in A$ such that $t_ix=gt_j$ and $At_ix=At_j$ (see the proof of \cite[Lemma 3.3]{KSS} and \cite[Section 5]{Lohrey} for further details).

\begin{theorem}\label{thm:CF}
Any context-free set $C$ in a finitely generated virtually abelian group $G$ is rational and, moreover, a rational expression for $C$ can be effectively computed.

\end{theorem}

\begin{proof}
Let $G$ be a finitely generated virtually abelian group with a finite-index free abelian normal subgroup $A$, and a finite transversal $T$, as in Definition \ref{def:VA}. Since $A$ is a subgroup of finite index it is a recognisable set, as are all cosets $At$ with $t \in T$. Let $X$ and $Y$ be finite generating sets of $G$ and $A$, respectively. Moreover, let $A_X$ be the preimage of $A$ in $X^*$, that is, $A_X=\pi^{-1}(A)$. The set $A_X$ is regular, and a finite state automaton $M_{A_X}$ over the alphabet $X$ can be effectively computed to produce $A_X$: we can construct the automaton for $\pi^{-1}(A)$ since we know a presentation for the quotient $\Delta=G/A$ (Definition \ref{def:VA}(3)) with respect to the generating set $\bar{X}=\{\bar{x}=xA : x \in X\}$; then $\pi^{-1}(A)$ is just the word problem of $\Delta$ with respect to $\bar{X}$, and the words representing the identity in $\Delta$ can be read off the finite Cayley graph of $\Delta$. In fact, the automaton $M_{A_X}$ can be taken to be $Cay(\Delta, \bar{X})$ and (re)using $x$ for every generator $\bar{x}$.

 Let $C$ be a context-free set in $G$ over $X$. 
 We write $C = \bigcup_{t \in T} (C \cap At)$, and it suffices to show that each $C \cap At$ is effectively rational over $X$. In fact, $C \cap At$ is effectively rational if and only if ($C \cap At)t^{-1}  = Ct^{-1} \cap A$ is. So it suffices to prove the result for $Ct \cap A$, $t\in T$. The set $Ct$ is context-free as the translate of $C$; that is, there is a context-free set $C_t \subset X^*$ explicitly given by a grammar $\Gamma_{C_t}$ in Chomsky normal form such that $\pi(C_t)=Ct$.


We claim that $Ct \cap A$ is effectively rational over $Y$, the generating set of $A$. To show this, let $L_{X, t}=C_t \cap A_X$; this is a context-free language over $X$ that can be effectively computed from $M_{A_X}$ and $\Gamma_{C_t}$, and such that $\pi(L_{X, t})=Ct \cap A$. By Lemma \ref{transducerKSS} there is an explicit, computable, rational transduction $\sigma\subseteq Y^* \times X^*$ such that $A_X=\sigma(A_Y)$, where $A_Y$ is the language of all words over $Y$ that represent elements in $A$. That is, $\sigma^{-1}: X^* \rightarrow Y^*$ is a partial function such that, for $w \in X^*$, $\sigma^{-1}(w)$ is defined if and only if $w$ represents an element of $A$, in which case $\sigma^{-1}(w)$ is an element of $Y^*$ representing the same element of $A$ as $w$. Moreover, $\sigma^{-1}(L_{X,t})$ produces a language, say $L_{Y,t}$, of words over $Y$ consisting of words that represent exactly the elements in $Ct\cap A$. Since being context-free is preserved by preimages of rational transductions, $L_{Y,t}$ is context-free and a grammar for it can be effectively computed since the rational transduction $\sigma$ is explicit.

 Now $L_{Y,t}$ is context-free and $\pi(L_{Y,t})$ and $\ab(L_{Y,t})$ represent the same subset of elements of $A$, namely $Ct\cap A$. By the strengthened version of Parikh's result (Theorem \ref{Parikh}), $\ab(L_{Y,t})$ is a (computably) semilinear set over $Y$, and a finite state automaton $M_{L_t}$ for a regular language $R_t$ over $Y$ with the same Parikh image is computable, therefore $Ct \cap A$ is effectively rational since it can be obtained from $M_{L_t}$. This proves the claim that $Ct \cap A$ is effectively rational over $Y$. Finally, we can apply the transduction $\sigma$ to $R_t \subset Y^*$ and get $Ct\cap A$ effectively rational over $X$, which proves therefore the theorem.
\end{proof}

%
%

\section{Order Constraints}

In this section we look at lexicographic order of solutions to equations, when these are written in the standard normal form. We then also consider shortlex order, using the length arguments from Section \ref{sec:length}. The lexicographic and shortlex order constraints compare
normal form words, as defined below.

\begin{definition}\label{def:orders}
	Let \(G\) be a virtually abelian group given as in Definition \ref{def:VA}. Fixing a finite-index free abelian
	normal subgroup \(A\) and a (finite) right transversal \(T\), we obtain the
	subgroup transversal normal form \(\eta\). For a fixed ordering on the free
	abelian basis of \(A\), and a fixed ordering on \(T\), we obtain a
	lexicographic ordering \(\leq_\text{lex}\) on the normal form words. By
	ordering the normal form words using shortlex instead, we obtain another
	ordering \(\leq_\text{shortlex}\).
\end{definition}

Note that both of these are dependent on the choice of normal subgroup and
transversal used to define \(\eta\), and the orderings on the free abelian
basis and transversal, and thus there will be multiple lexicographic and
shortlex orderings on a given virtually abelian group.

\begin{remark}
  Lexicographic or shortlex are total orders on the elements of a virtually
  abelian group, but are not left (or right) invariant orders, i.e. they are
  not invariant under left (or right) multiplication. For the virtually abelian
  groups with infinitely many left orders there exist left-invariant orders for
  which comparing two elements is undecidable \cite{AntolinRivasSu}, so using
  left-invariant orders is not a feasible constraint when considering the
  Diophantine Problem.
\end{remark}

Before we can prove the main result of this section, we must show that we can effectively construct intersections of rational subsets of virtually abelian groups. We need the following result of Grunschlag.
\begin{lemma}\cite[Corollary 2.3.8]{Grunschlag_thesis}\label{lem:Grunschlag}
	Let $G$ be a group with finite generating set $S$ and $H$ a finite index subgroup of $G$ with right transversal $T$. For each rational subset $U\subset G$ such that $U\subset Ht$ for some $t\in T$, there exists an effectively constructible rational subset $V\subset H$ such that $U=Vt$.
\end{lemma}
In \cite{CiobanuEvetts}, rational subsets of virtually abelian groups are characterised as certain kinds of semilinear subsets, which implies that their intersections are also rational. The following lemma makes this effective.
\begin{lemma}\label{lem:rationalintersection}
	Let $G$ be a finitely generated virtually abelian group. Given finitely many rational subsets of $G$, their intersection is an effectively constructible semilinear set.
\end{lemma}
\begin{proof}
	Let $S$ be a finite generated set for $G$ and let $A$ denote a free abelian normal subgroup of $G$ of finite index and $T$ a choice of right transversal. Since $A$ has finite index in $G$, any coset $At$ is \emph{recognisable}. In other words, the full preimage of $At$ in $S^*$ is a regular language (see \cite[Proposition 6.3]{HerbstThomas}).
	
	Suppose that $U$ is a rational subset of $G$ with $R\subset S^*$ a regular language whose image is $U$. Then intersecting $R$ with the full preimage of any coset $At$ gives a new regular language $R'$ whose image is $U\cap tA$. Since intersections of regular languages can be found algorithmically \cite{groups_langs_aut}, each intersection $U\cap At$ is an effectively constructible rational set. Now since $U\cap At\subset At$, Lemma \ref{lem:Grunschlag} implies that we can find a rational subset $V_t\subset A$ such that $U\cap At =V_t t$, for each $t\in T$ (and so $U=\bigcup V_t t$).
	
  In \cite[Section 2]{RationalSets} it is shown that in any commutative monoid,
  rational sets and semilinear sets coincide and, moreover, given a rational
  expression for a set, a semilinear expression can be found. Thus, we can
  effectively express each rational subset $V_t\subset A$ as a semilinear
  subset of $A$, so that 
  \[
    U
    = \bigcup_{t\in
    T} \bigcup_{i=1}^{k_t} (a_{t,i}+B_{t,i}^*)t
  \] 
  for some $a_{t,i}\in A$ and finite $B_{t,i}\subset A$ (where we write the
  group operation in $A$ additively). Now, since intersection distributes over
  union (and in light of Remark \ref{rem:semilinearunion}), the intersection of
  two (and hence finitely many) rational sets $U,U'\subset G$ is an effectively
  constructible semilinear (and hence rational) set as long as the intersection
  of any two linear sets of the form $(a_{t,i}+B_{t,i}^*)t$ is an effectively
  constructible semilinear set. Consider two such sets $(a+B^*)t$ and
  $(c+D^*)s$, for $t,s\in T$, $a,c\in A$, and finite $C,D\subset A$. We have
	\begin{align*}
		(a+B^*)t \cap (c+D^*)s=\begin{cases} \emptyset & t\neq s \\ ((a+B^*)t\cap (c+D^*)) & t=s \end{cases}.
	\end{align*}
	By Theorem \ref{thm:intersectionsemi}, the intersection $(a+B^*)\cap (c+D^*)$ is semilinear and effectively constructible and hence so is its translate by $t$, finishing the proof.
\end{proof}

We now show that lexicographic order constraints are equivalent to a collection
of rational constraints. Note that two lexicographically
comparable words in the subgroup transversal normal form share a common prefix.

\begin{theorem}
	\label{thm:lexicographic}
	Let \(G\) be a finitely generated virtually abelian group. Fix a finite-index
	free abelian subgroup \(A\), a free abelian basis \(B=\{a_1,  \ldots, a_n\}\) with $a_1< a_1^{-1} \ldots < a_n < a_n^{-1}$ for \(A\), and a
	(finite) right transversal \(T\) also equipped with an order \(\leq\).

	Then verifying a lexicographic order constraint \(X
	\leq_\textnormal{lex} Y\) is equivalent to checking membership of each of \(X\), \(Y\) and \(X^{-1}Y\)
	in a certain effectively constructible rational set.
\end{theorem}

\begin{proof}
	We write all group elements in (subgroup transversal) normal form \(a_1^{i_1}
	\cdots a_n^{i_n} s\), where $a_i \in B$ and $s \in T$, and note that
	\(a_1^{i_1} \cdots a_n^{i_n} s \leq_\text{lex} a_1^{j_1} \cdots a_n^{j_n} t\) implies one
	of: (1) \(i_1 = j_1, \ \ldots, \ i_n = j_n\) and \(s \leq t\), or (2) there
	exists \(p \in \{1, \ldots, n\}\) such that \(i_1 = j_1, \ldots, i_{p - 1} =
	j_{p - 1}\) and \(i_p < j_p\).

	Case 1: there exists \(p \in \{1, \ldots, n\}\) such that \(i_1 = j_1, \ldots,
	i_{p - 1} = j_{p - 1}\) and \(i_p < j_p\).\\
	This case is a finite disjunction across \(p \in \{1, \ldots n\}\) of the case
	where for this fixed \(p\), \(i_1 = j_1, \ldots, i_{p - 1} = j_{p - 1}\) and
	\(i_p < j_p\). Thus it suffices to show each of these cases reduces to
	checking membership in effectively constructible rational sets. So fix such a \(p \in \{1,
	\ldots, n\}\).

  We first show that checking \(i_q = j_q\) for all \(q < p\) can be done by
  checking membership of \((a_1^{i_1} \cdots a_n^{i_n} s)^{-1} (a_1^{j_1}
  \cdots a_n^{j_n} t)\) in a rational set. Let \(\phi_s, \phi_t \in \Aut(A)\)
  be the automorphisms induced by conjugation by \(s\) and \(t\), respectively.
  We will show that verifying that \(i_q = j_q\) for all \(q < p\) is equivalent to checking
  membership in the set \(\phi^{-1}_s(\pi((a_p^\pm)^\ast \cdots (a_n^\pm)^\ast)
  )) \pi(s^{-1} t)\).  If \(i_q = j_q\) for all \(q < p\), then
	\begin{align*}
		(a_1^{i_1} \cdots a_n^{i_n} s)^{-1} (a_1^{j_1} \cdots a_n^{j_n}
		t) & = s^{-1} a_1^{j_1 - i_1} \cdots a_n^{j_n - i_n} t \\
		& = s^{-1} a_p^{j_p - i_p} \cdots a_n^{j_n - i_n} t \\
		& = \phi_s^{-1} (a_p^{j_p - i_p} \cdots a_n^{j_n - i_n})  s^{-1} t
		\in \phi^{-1}_s(\pi((a_p^\pm)^\ast \cdots
    (a_n^\pm)^\ast) ))\pi(s^{-1} t).
	\end{align*}
  Conversely, suppose \((a_1^{i_1} \cdots a_n^{i_n} s)^{-1} (a_1^{j_1} \cdots
  a_n^{j_n} t) \in \phi^{-1}_s(\pi((a_p^\pm)^\ast \cdots
  (a_n^\pm)^\ast) ))\pi(s^{-1} t)\). Then
	\begin{align*}
		(a_1^{i_1} \cdots a_n^{i_n} s)^{-1} (a_1^{j_1} \cdots
		a_n^{j_n} t) & = s^{-1} a_1^{j_1 - i_1} \cdots a_n^{j_n - i_n} t \\
		& =  \phi_s^{-1} (a_1^{j_1 - i_1} \cdots a_n^{j_n - i_n}) s^{-1} t
	\end{align*}
	This lies in \(\phi^{-1}_s(\pi((a_p^\pm)^\ast \cdots
  (a_n^\pm)^\ast) ))\pi(s^{-1} t)\) if and only if \(j_1
	- i_1 = \cdots = j_{p - 1} - i_{p - 1} = 0\), as required.

  If \(x, y \in G\) are arbitrary elements, then \(x\) and \(y\) are of the
  above form; that is \(x = a_1^{i_1} \cdots a_n^{i_n} s\) and \(y = a_1^{j_1}
  \cdots a_n^{j_n} t\) for some \(i_1, \ldots, i_n, j_1, \ldots, j_n \in
  \mathbb{Z}\) and \(s, t \in T\) with \(i_q = j_q\) for all \(q < p\). We have
  that \(i_q = j_q\) for all \(q < p\) if and only if \(xy^{-1}\) lies in the
  set \(\phi^{-1}_s(\pi((a_p^\pm)^\ast \cdots (a_n^\pm)^\ast) ))\pi(s^{-1}
  t)\). However, as this set's definition depends on \(s\) and \(t\), this set
  depends on \(x\) and \(y\). So what we must do is union across all
  possibilities of \(s^{-1} t\), which will be \(T^{-1}T\).  Thus \(x\) and
  \(y\) satisfy the property that \(i_q = j_q\) for all \(q < p\) if and only
  if \(x y^{-1} \in \phi^{-1}_s(\pi((a_p^\pm)^\ast \cdots (a_n^\pm)^\ast) ))
  T^{-1}T\).

  Note that \(\phi^{-1}_s(\pi((a_p^\pm)^\ast \cdots (a_n^\pm)^\ast) ))
  T^{-1}T\) is rational as the image of a rational set under an automorphism is
  effectively rational, and the concatenation of rational sets is rational (see Section \ref{sec:FL}).

	It remains to show that \(i_p < j_p\) can be determined by checking membership
	in an effectively constructible rational set. Let \(k\) be the exponent of \(a_p\) within
  \(\phi_s(s^{-1} t)\). We will show that \(i_p < j_p\) if
	and only if \((a_1^{i_1} \cdots a_n^{i_n} s)^{-1} (a_1^{j_1} \cdots a_n^{j_n}
	t) \in \phi_s^{-1}(\pi(a_1^\pm)^\ast \cdots (a_{p - 1}^\pm)^\ast (a_p^{k + 1} a_p^\ast)
	(a_{p + 1}^\pm)^\ast \cdots (a_n^\pm)^\ast) T \). We have
	\begin{align*}
		(a_1^{i_1} \cdots a_n^{i_n} s)^{-1} (a_1^{j_1} \cdots a_n^{j_n}
		t) & = \phi_s^{-1} (a_1^{j_1 - i_1} \cdots a_n^{j_n - i_n})
    s^{-1} t \\
		& = \phi_s^{-1} (a_1^{j_1 - i_1} \cdots a_n^{j_n - i_n}
    \phi_s (s^{-1} t) ).
	\end{align*}
  Thus the exponent of \(a_p\) in the expression for \(a_1^{j_1 - i_1} \cdots
  a_n^{j_n - i_n} \phi_s (s^{-1} t)\) will equal \(j_p - i_p + k\), and so
  \(i_p < j_p\) if and only if this exponent is strictly greater than \(k\) and
  so checking \(i_p < j_p\) is equivalent to checking membership in
  \(\phi_s^{-1}(\pi(a_1^\pm)^\ast \cdots (a_{p - 1}^\pm)^\ast (a_p^{k + 1}
  a_p^\ast) (a_{p + 1}^\pm)^\ast \cdots (a_n^\pm)^\ast )T \), which is rational
  as the automorphic image of a rational set. We have thus shown that \(x
  \leq_\text{lex} y\) and we are in this case (Case 1) if and only if \(xy^{-1}
  \in \phi^{-1}_s(\pi((a_p^\pm)^\ast \cdots (a_n^\pm)^\ast) ))
  T^{-1}T\) and \(xy^{-1} \in \phi_s^{-1}(\pi(a_1^\pm)^\ast \cdots (a_{p -
  1}^\pm)^\ast (a_p^{k + 1} a_p^\ast) (a_{p + 1}^\pm)^\ast \cdots
  (a_n^\pm)^\ast) T\), and so checking we lie in this case and that \(x
  \leq_\text{lex} y\) is equivalent to verifiying membership of \(x y^{-1}\)
  in a finite intersection of effectively constructible rational sets. By Lemma \ref{lem:rationalintersection}, this is equivalent to verifying membership in a single effectively constructible rational set.

	Case 2: \(i_1 = j_1, \ldots, i_n = j_n\) and \(s \leq_\text{lex} t\). \\
  We can use the argument from the first part of Case 1 to check that \(i_1 =
  j_1, \ldots, i_n = j_n\), by just taking \(q\) to be \(n + 1\).

  Note that for each \(r \in T\), the coset \(Ar\) is rational and explicitly
  computable: it is the projection of the language \((a_1^\pm)^\ast \cdots
  (a_n^\pm)^\ast r\). As \(T\) is finite, checking \(g \leq_\text{lex} h\) in
  this case reduces to checking if \(g \in Ar\) and \(h \in Ar'\) for any (of
  the finitely many) \((r, \ r') \in T \times T\), with \(r \leq_\text{lex}
  r'\).
\end{proof}

Since one element is shortlex smaller than another if and only if it is length
smaller, or length equal and lexicographically smaller, shortlex constraints
reduce to checking length constraints and lexicographic constraints, both of
which are equivalent to checking effectively constructible rational sets.

\begin{corollary}
	Let \(G\) be a finitely generated virtually abelian group. Fix a finite-index
	free abelian subgroup \(A\), a free abelian basis \(B\) for \(A\) and a
	(finite) right transversal \(T\). Then verifying any shortlex order constraint \(X
	\leq_\text{shortlex}Y\) is equivalent to checking membership of each of \(X\), \(Y\) and \(X^{-1}Y\)
	in a certain effectively constructible rational set.
\end{corollary}

\begin{proof}
  We have that \(X\) is shortlex less than or equal to \(Y\) if and only if
  \(|X| < |Y|\), or \(|X| = |Y|\) and \(X \leq_\text{lex} Y\). We can check
  both \(|X| < |Y|\) and \(|X| = |Y|\) using Theorem~\ref{thm:lengthrational}
  (with each element of the generating set \(B\cup T\) weighted with a $1$),
  and we can check if \(X \leq_\text{lex} Y\) using
  Theorem~\ref{thm:lexicographic}. We can put together two conditions, such as
  \(|X| = |Y|\) and \(X \leq_\text{lex} Y\), by checking membership in all the
  rational sets corresponding to these constraints produced effectively in
  Theorems~\ref{thm:lengthrational} and \ref{thm:lexicographic}, and using
  Lemma~\ref{lem:rationalintersection} to resolve any intersections.
\end{proof}
\section*{Acknowledgements}
The first named author acknowledges a Scientific Exchanges grant (number IZSEZ0 213937) of the Swiss National Science Foundation. The second and third named authors thank the Heilbronn Institute for Mathematical Research for support during this work.

The first named author would like to thank Markus Lohrey and Georg Zetzsche for helpful discussions and references concerning Section \ref{sec:CF}.

\bibliography{references}{}

\begin{thebibliography}{10}

\bibitem{Abdulla}
Parosh~Aziz Abdulla, Mohamed~Faouzi Atig, Yu-Fang Chen, Luk{\'a}{\v{s}}
  Hol{\'i}k, Ahmed Rezine, Philipp R{\"u}mmer, and Jari Stenman.
\newblock String constraints for verification.
\newblock In Armin Biere and Roderick Bloem, editors, {\em Computer Aided
  Verification}, pages 150--166, Cham, 2014. Springer International Publishing.

\bibitem{Amadini}
Roberto Amadini.
\newblock A survey on string constraint solving.
\newblock {\em ACM Comput. Surv.}, 55(1), nov 2021.

\bibitem{AntolinRivasSu}
Yago Antol\'{\i}n, Crist\'{o}bal Rivas, and Hang~Lu Su.
\newblock Regular left-orders on groups.
\newblock {\em J. Comb. Algebra}, 6(3-4):265--314, 2022.

\bibitem{Benson}
M.~Benson.
\newblock Growth series of finite extensions of {${\bf Z}^{n}$}\ are rational.
\newblock {\em Invent. Math.}, 73(2):251--269, 1983.

\bibitem{Berstel}
Jean Berstel.
\newblock {\em Transductions and context-free languages.}
\newblock B. G. Teubner, Stuttgart, 1979.

\bibitem{Bishop}
A.~Bishop.
\newblock Geodesic growth in virtually abelian groups.
\newblock {\em J. Algebra}, 573:760--786, 2021.

\bibitem{RichardBuchi1988}
J.~Richard B\"uchi and Steven Senger.
\newblock Definability in the existential theory of concatenation and
  undecidable extensions of this theory.
\newblock {\em Zeitschrift fur mathematische Logik und Grundlagen der
  Mathematik}, 34(4):337--342, 1988.

\bibitem{Carvahlo}
A.~Carvahlo.
\newblock Algebraic and context-free subsets of subgroups.
\newblock 2022.
\newblock https://arxiv.org/abs/2210.10001.

\bibitem{eqns_hyp_grps}
L.~Ciobanu and M.~Elder.
\newblock The complexity of solution sets to equations in hyperbolic groups.
\newblock {\em Israel J. Math.}, 245(2):869--920, 2021.

\bibitem{CiobanuEvetts}
L.~Ciobanu and A.~Evetts.
\newblock Rational sets in virtually abelian groups: languages and growth.
\newblock {\em L'Enseignement math\'ematique}, 2023.
\newblock https://arxiv.org/abs/2205.05621.

\bibitem{CiobanuGarreta}
L.~Ciobanu and A.~Garrett.
\newblock Group equations with abelian predicates.
\newblock {\em Int. Math. Res. Notices}, 2023+.
\newblock https://arxiv.org/abs/2204.13946.

\bibitem{CiobanuLevine}
L.~Ciobanu and A.~Levine.
\newblock Languages, groups and equations.
\newblock 2023.
\newblock https://arxiv.org/abs/2303.07825.

\bibitem{Dahmani}
F.~Dahmani.
\newblock Existential questions in (relatively) hyperbolic groups.
\newblock {\em Israel J. Math.}, 173:91--124, 2009.

\bibitem{dahmani_guirardel}
Fran\c{c}ois Dahmani and Vincent Guirardel.
\newblock Foliations for solving equations in groups: free, virtually free, and
  hyperbolic groups.
\newblock {\em J. Topol.}, 3(2):343--404, 2010.

\bibitem{DAIV}
Flavio D'Alessandro, Benedetto Intrigila, and Stefano Varricchio.
\newblock Quasi-polynomials, linear {D}iophantine equations and semi-linear
  sets.
\newblock {\em Theoret. Comput. Sci.}, 416:1--16, 2012.

\bibitem{DayManeaWE}
Joel~D Day, Vijay Ganesh, Paul He, Florin Manea, and Dirk Nowotka.
\newblock The satisfiability of word equations: Decidable and undecidable
  theories.
\newblock In {\em International Conference on Reachability Problems}, pages
  15--29. Springer, 2018.

\bibitem{Dickson}
Leonard~Eugene Dickson.
\newblock Finiteness of the {O}dd {P}erfect and {P}rimitive {A}bundant
  {N}umbers with {$n$} {D}istinct {P}rime {F}actors.
\newblock {\em Amer. J. Math.}, 35(4):413--422, 1913.

\bibitem{RationalSets}
S.~Eilenberg and M.~P. Sch\"{u}tzenberger.
\newblock Rational sets in commutative monoids.
\newblock {\em J. Algebra}, 13:173--191, 1969.

\bibitem{ParikhConstructive}
Javier Esparza, Pierre Ganty, Stefan Kiefer, and Michael Luttenberger.
\newblock Parikh's theorem: a simple and direct automaton construction.
\newblock {\em Inform. Process. Lett.}, 111(12):614--619, 2011.

\bibitem{Evetts}
A.~Evetts.
\newblock Rational growth in virtually abelian groups.
\newblock {\em Illinois J. Math.}, 63(4):513--549, 2019.

\bibitem{EvettsLevine}
A.~Evetts and A.~Levine.
\newblock Equations in virtually abelian groups: {L}anguages and growth.
\newblock {\em Internat. J. Algebra Comput.}, 32(3):411--442, 2022.

\bibitem{GarretaGray}
Albert Garreta and Robert Gray.
\newblock On equations and first-order theory of one-relator monoids.
\newblock {\em Inform. and Comput.}, 281, December 2021.

\bibitem{GinsburgSpanier}
Seymour Ginsburg and Edwin~H. Spanier.
\newblock Bounded {${\rm ALGOL}$}-like languages.
\newblock {\em Trans. Amer. Math. Soc.}, 113:333--368, 1964.

\bibitem{Grunschlag_thesis}
Z.~Grunschlag.
\newblock {\em Algorithms in geometric group theory}.
\newblock PhD thesis, University of California, Berkeley, 1999.

\bibitem{Herbst}
Thomas Herbst.
\newblock On a subclass of context-free groups.
\newblock {\em RAIRO Inform. Th\'{e}or. Appl.}, 25(3):255--272, 1991.

\bibitem{HerbstThomas}
Thomas Herbst and Richard~M. Thomas.
\newblock Group presentations, formal languages and characterizations of
  one-counter groups.
\newblock {\em Theoret. Comput. Sci.}, 112(2):187--213, 1993.

\bibitem{groups_langs_aut}
D.~F. Holt, S.~Rees, and C.~E. R\"{o}ver.
\newblock {\em Groups, languages and automata}, volume~88 of {\em London
  Mathematical Society Student Texts}.
\newblock Cambridge University Press, Cambridge, 2017.

\bibitem{KSS}
Mark Kambites, Pedro~V. Silva, and Benjamin Steinberg.
\newblock On the rational subset problem for groups.
\newblock {\em J. Algebra}, 309(2):622--639, 2007.

\bibitem{Levine}
Alex Levine.
\newblock E{DT}0{L} solutions to equations in group extensions.
\newblock {\em J. Algebra}, 619:860--899, 2023.

\bibitem{Lohrey}
Markus Lohrey.
\newblock The rational subset membership problem for groups: a survey.
\newblock In {\em Groups {S}t {A}ndrews 2013}, volume 422 of {\em London Math.
  Soc. Lecture Note Ser.}, pages 368--389. Cambridge Univ. Press, Cambridge,
  2015.

\bibitem{Mihailova}
K.~A. Miha\u{\i}lova.
\newblock The occurrence problem for direct products of groups.
\newblock {\em Mat. Sb. (N.S.)}, 70 (112):241--251, 1966.

\bibitem{NeumannShapiro95}
W.~D. Neumann and M.~Shapiro.
\newblock Automatic structures, rational growth, and geometrically finite
  hyperbolic groups.
\newblock {\em Invent. Math.}, 120(2):259--287, 1995.

\bibitem{NeumannShapiro97}
W.~D. Neumann and M.~Shapiro.
\newblock Regular geodesic normal forms in virtually abelian groups.
\newblock {\em Bull. Austral. Math. Soc.}, 55(3):517--519, 1997.

\end{thebibliography}
\bibliographystyle{plain}

\end{document}